\documentclass[leqno]{amsart}
\usepackage{amsmath,amsthm,amssymb}
\usepackage[T1]{fontenc}

\newtheorem{Th}{Theorem}[section]
\newtheorem{Prop}[Th]{Proposition}
\newtheorem{Lemma}[Th]{Lemma}

\theoremstyle{remark}
\newtheorem{Remark}[Th]{Remark}
\newtheorem*{Example}{Example}
\newtheorem{Definition}{Definition}
\newtheorem*{ratner}{Ratner's property}

\newcommand{\cb}{\mathcal{B}}

\newcommand{\cc}{\mathcal{C}}

\newcommand{\ct}{\mathcal{T}}
\newcommand{\cs}{\mathcal{S}}

\newcommand{\ot}{\otimes}
\newcommand{\la}{\lambda}

\newcommand{\vep}{\varepsilon}

\newcommand{\rat}{\operatorname{R}}

\newcommand{\R}{{\mathbb{R}}}
\newcommand{\T}{{\mathbb{T}}}
\newcommand{\cir}{{\mathbb{S}^1}}
\newcommand{\C}{{\mathbb{C}}}
\newcommand{\Z}{{\mathbb{Z}}}
\newcommand{\N}{{\mathbb{N}}}
\newcommand{\Q}{{\mathbb{Q}}}
\newcommand{\xbm}{(X,\mathcal{B},\mu)}

\newcommand{\ycn}{(Y,\mathcal{C},\nu)}

\newcommand{\var}{\operatorname{Var}}
\newcommand{\bez}{\nopagebreak\hspace*{\fill}\nolinebreak$\Box$}
\begin{document}

\title[Ratner's property and mixing for special
flows]{Ratner's property and  mixing for special flows over
two--dimensional rotations}
\author{K.\ Fr\k{a}czek \and M.\
Lema\'nczyk}
\address{K.\ Fr\k{a}czek \\ Faculty of Mathematics
and Computer Science\\ Nicolaus Copernicus University\\ ul.
Chopina 12/18, 87-100 Toru\'n, Poland}
\address{M.\ Lema\'nczyk\\
Institute of Mathematics of Polish Academy of Sciences, ul.\
\'Sniadeckich 8, 00-950 Warszawa, Poland and Faculty of
Mathematics and Computer Science\\ Nicolaus Copernicus
University\\ ul. Chopina 12/18, 87-100 Toru\'n, Poland}

\email{fraczek@mat.umk.pl, mlem@mat.umk.pl}

\thanks{Research partially
supported by MNiSzW grant N N201 384834 and Marie Curie ``Transfer
of Knowledge'' program, project MTKD-CT-2005-030042 (TODEQ)}

\begin{abstract}We consider  special flows over
two-dimensional rotations by $(\alpha,\beta)$ on $\T^2$ and under
piecewise $C^2$ roof functions $f$ satisfying  von Neumann's
condition $$\int_{\T^2}f_x(x,y)\,dx\,dy\neq 0\neq
\int_{\T^2}f_y(x,y)\,dx\,dy.$$ Such flows are shown to be always
weakly mixing and never partially rigid. For an uncountable set of
$(\alpha,\beta)$ with both $\alpha$ and $\beta$ of unbounded
partial quotients the strong mixing property is  proved to hold.
It is also proved that while specifying to a subclass of roof
functions and to ergodic rotations for which $\alpha$ and $\beta$
are of bounded partial quotients  the corresponding special flows
enjoy so called weak Ratner's property. As a consequence, such
flows turn out to be mildly mixing.
\end{abstract}

\maketitle

\section{Introduction} Mixing properties, especially
strong and mild mixing, of special flows over one- and
multi-dimensional irrational rotations under some regular roof
functions  have  been intensively studied during last few years,
e.g.\ \cite{Fa2}--\cite{Fa3}, \cite{Fr-Le}, \cite{Fr-Le-Le},
\cite{Kh-Si}, \cite{Ko0}--\cite{Le}. Such special flows appear
often  while studying smooth flows (or at least ergodic components
of smooth  flows) on some compact manifolds; indeed,  a choice of
a natural transversal may lead to a special representation over a
rotation, see e.g.\ \cite{Ar}, \cite{Fr-Le}, \cite{Ka-Ha},
\cite{Ko76}.

It is already in 1932 when von Neumann \cite{vonN} considered
special flows over  irrational rotations on $\T=[0,1)$ under  roof
functions $f$ which were piecewise $C^1$. He proved weak mixing of
such flows whenever the condition
\begin{equation}\label{vN0}\int_{\T}f'(x)\,dx\neq0\end{equation}
was satisfied. Linear functions $f(x)=ax+b$ for $0\leq x<1$ (with
$a\neq0$ and $b\in\R$ so that $f>0$) are the simplest examples of
roof functions satisfying von Neumann's condition~(\ref{vN0}).
Piecewise $C^1$--functions are of bounded variation, hence, as
shown by Kochergin \cite{Ko0} in 1972, the corresponding special
flows are not mixing. A natural question whether a special flow
over an irrational rotation by $\alpha\in[0,1)$ under $f$
piecewise $C^1$ and satisfying~(\ref{vN0}) can enjoy a stronger
property than weak mixing  was answered positively
in~\cite{Fr-Le}; indeed, such flows  turn out to be mildly mixing
whenever $\alpha$ has bounded partial quotients. As a matter of
fact, the mild mixing property has been proved in \cite{Fr-Le} in
two independent steps: first, the absence of partial rigidity
(which does not require any Diophantine condition on $\alpha$) has
been proved and then so called Ratner's property has been
established for $\alpha$ with bounded partial quotients.

In the present paper we consider special flows over an egodic
two-dimensional rotation  $T(x,y)=(x+\alpha,y+\beta)$. Our roof
functions $f:\T^2\to\R_+$ will be piecewise $C^2$ (discontinuities
of $f$ are contained in finitely many horizontal and vertical
lines, see Definition~\ref{defpiec}) and will satisfy a
two-dimensional analog of (\ref{vN0})

\begin{equation}
\label{vonNeumann02}
\int_{\T^2}f_x(x,y)\,dx\,dy\neq 0\text{ or }
\int_{\T^2}f_y(x,y)\,dx\,dy\neq 0.
\end{equation}
In what follows (\ref{vonNeumann02}) will be referred to as the
{\em weak von-Neumann's condition}. We will observe that this
condition implies the weak mixing property of the corresponding
special flows $T^f$ (Theorem~\ref{slabemieszanie}) as well as the
absence of partial rigidity (Theorem~\ref{aopr}). As in
\cite{Fr-Le}, our aim will be to prove that such flows are mildly
mixing.  If we want the strategy from \cite{Fr-Le} of showing the
mild mixing property (under some Diophantine assumptions on
$(\alpha,\beta)$) to work we need to prove an analog of Ratner's
property for such flows. This is done only partially, namely, in a
restricted class of roof functions satisfying (\ref{vonNeumann02})
and both $\alpha$ and $\beta$ are assumed to have bounded partial
quotients, see Theorem~\ref{specialR} in which so called weak
Ratner's property is proved to hold. The class of roof functions
includes all positive linear functions $f(x,y)=ax+by+c$ with
$a/b\in\R\setminus\Q$. Then, the mild mixing property follows
(Theorem~\ref{mmixing}). Proving (even the weak) Ratner's property
of such flows is of independent interest, as it has some other
ergodic consequences (Theorem~\ref{prpr}, see also~\cite{Th}).
Recall that the original notion, introduced by Ratner in \cite{Ra}
and called there ${\mathcal H}_p$-property, is as follows:
\begin{ratner}
Let $(X,d)$ be a $\sigma$--compact metric space, $\mu$ a
probability Borel
 measure on $(X,d)$ and $(S_t)_{t\in\R}$ a
$\mu$--preserving flow. The flow $(S_t)_{t\in\R}$ is called
$\mathcal{H}_p$--flow, $p\neq 0$, if for every $\vep>0$ and
$N\in\N$ there exist $\kappa=\kappa(\vep)>0$,
$\delta=\delta(\vep,N)>0$ and a Borel subset $Z=Z(\vep,N)\subset
X$ with $\mu(Z)>1-\vep$  such that if $x,x'\in Z$, $x'$ is not in
the orbit of $x$ and $d(x,x')<\delta$, then there are
$M=M(x,x')\geq N$, $L=L(x,x')\geq N$ with $L/M\geq\kappa$  such
that if we denote
\[K^\pm=\{n\in\Z\cap[M,M+L]:d(S_{np}(x),S_{(n\pm 1)p}(x'))<\vep\}\]
then either $\#K^+/L>1-\vep$ or $\#K^-/L>1-\vep$.
\end{ratner}
Ratner's property, originally proved by M. Ratner \cite{Ra} for
horocycle flows, in the framework of special flows over irrational
rotations first appeared in \cite{Fr-Le}. In fact, already in
\cite{Fr-Le} the original definition of Ratner has been modified
and $\pm p$ was replaced by a finite subset of $\R\setminus\{0\}$.
In the present paper we need a further weakening of the
definition: we introduce a compact set $P\subset\R\setminus\{0\}$
so that the orbits of two close different points are close up to a
shift of time belonging to $P$ on  sufficiently long pieces of
orbits. We call this property {\em weak Ratner's property} (see
Definition~\ref{defrat}).

Unlike the one-dimensional rotation case, special flows over
two-dimensional rotations even under smooth functions can be
mixing, see  \cite{Fa1},  \cite{Fa3}. In
Section~\ref{strongmixing} we show that special flows with
piecewise $C^2$ roof functions and satisfying the following {\em
strong von Neumann's condition}
\begin{equation}\label{vonNeumann03}
\int_{\T^2}f_x(x,y)\,dx\,dy\neq 0\text{ and }
\int_{\T^2}f_y(x,y)\,dx\,dy\neq 0\end{equation} are mixing for
uncountably many $(\alpha,\beta)\in\T^2$ (Theorem~\ref{smixing}).
The main tool to prove mixing property we use is a Fayad's
criterion from \cite{Fa1}. In particular, in the linear case
$f(x,y)=ax+by+c$ mixing is possible for a special choice of
$\alpha,\beta$ -- a phenomenon which can not happen in the
one-dimensional case.

\subsection{Plan of the paper}The plan on the paper is as follows:
Section~\ref{secwprow} introduces terminology and notation that
will be used throughout the remainder of the paper. In
Section~\ref{secweakm} we will show weak mixing of the special
flow $T^f$ (Theorem~\ref{slabemieszanie}) assuming that the roof
function $f:\T^2\to\R_+$ is piecewise $C^2$ and satisfies
(\ref{vonNeumann02}). In Section~\ref{secabsrig} we will establish
the absence of partial rigidity under the same assumption
(Theorem~\ref{aopr}). The proofs of these results are proved in
spirit to the one--dimensional case in \cite{Fr-Le}.

Next part of the paper deals with mild mixing. We use a criterion
from \cite{Fr-Le}: If a flow is not partially rigid and it is a
finite extension of each of its non-trivial factors (finite fibers
factor property) then it is mildly mixing. The absence of partial
rigidity being already established, in order to deal with the
second assumption the notion of weak Ratner's property is
introduced in Section~\ref{secweakrat}. Then, in
Theorem~\ref{prpr}, it is proved that weak Ratner's property
implies finite fibers factor property.

In Section~\ref{secwrspf} we present techniques
(Lemma~\ref{generalrat} and Proposition~\ref{propklu}) that help
us in proving the weak Ratner property to hold for special flows
built over rotations. In Section~\ref{secconcrspf} we introduce a
class of piecewise $C^2$ von Neumann roof functions on $\T^2$ and
we consider the corresponding special flows over ergodic rotations
whose both coordinates have bounded partial quotients. Using
techniques from Section~\ref{secwrspf} for this class of special
flows, we prove weak Ratner's property (see
Theorem~\ref{specialR}), which finally establishes mild mixing.
Moreover, in Section~\ref{MMM} we provide an example from this
class  which is mildly mixing but is not mixing.

Section~\ref{strongmixing} deals with mixing property for special
flows with piecewise $C^2$ roof functions satisfying strong von
Neumann's condition (\ref{vonNeumann03}) and it uses methods
different from  earlier sections. We first notice that Fayad's
criterion \cite{Fa1} (alternating uniform stretch of the Birkhoff
sums in the vertical and horizontal directions) of mixing of
special flows for $C^2$ roof functions can be extended to
piecewise $C^2$ case. Then we  prove  mixing over an uncountable
family of rotations by $(\alpha,\beta)$ on $\T^2$ (both $\alpha$
and $\beta$ have unbounded partial quotients).

We will discuss some other consequences of the results proved in
the paper as well as some open problems in Section~\ref{concl}.

Our special thanks go to A.\ Katok who was the first to conjecture
that already linearity  over two dimensional rotations may be
sufficient for strong mixing property of the corresponding special
flows. Such mixing flows are apparently the simplest examples of
mixing flows in the framework of special flows under regular roof
functions and over multi-dimensional rotations.

We also thank both referees for numerous comments and suggestions
which led both to a better presentation as well as to stronger
results than in the first version of the paper. Especially, we
thank one of the referees for proposing the main idea of the proof
of Theorem~\ref{specialR}.

\section{Notation}\label{secwprow}
Let $T$ be an ergodic automorphism of  a standard probability
Borel space $(X,\mathcal{B},\mu)$, this is for every
$T$--invariant set $A\in\mathcal{B}$, either $A$ or its complement
$X\setminus A$ has measure zero. Assume $f:X\to\R$ is a strictly
positive integrable function and let $\mathcal{ B}(\R)$ and
$\la_{\R}$ denote Borel $\sigma$--algebra and Lebesgue measure on
$\R$ respectively. Then by $\mathcal{T}^f=(T^f_t)_{t\in\R}$ we
will mean the corresponding special flow under $f$ (see e.g.\
\cite{Co-Fo-Si}, Chapter 11) acting on $(X^f,\mathcal{
B}^f,\mu^f)$, where $X^f=\{(x,s)\in X\times \R:\:0\leq s<f(x)\}$
and $\mathcal{ B}^f$ $(\mu^f)$ is the restriction of $\mathcal{
B}\otimes\mathcal{ B}(\R)$ $(\mu\otimes\lambda_\R)$ to $X^f$.
Under the action of the flow $\mathcal{T}^f$ each point in $X^f$
moves vertically at unit speed, and we identify the point
$(x,f(x))$ with $(Tx,0)$.  Given $m\in\Z$ we put
\[f^{(m)}(x)=\left\{\begin{array}{ccc}
f(x)+f(Tx)+\ldots+f(T^{m-1}x) & \mbox{if} & m>0\\ 0 &
\mbox{if} &
m=0\\ -\left(f(T^mx)+\ldots+f(T^{-1}x)\right)  & \mbox{if} & m<0.
\end{array}\right.\]
Then  for every $(x,s)\in X^f$ we have
\[T^f_t(x,s)=(T^nx,s+t-f^{(n)}(x)),\]
where $n\in\Z$ is  unique such that $f^{(n)}(x)\leq
s+t<f^{(n+1)}(x)$.

If $X$ is equipped with a metric $d$ whose Borel $\sigma$--algebra
is equal to $\mathcal{B}$ then we will consider on $X^f$ the
metric $d^f$ defined by
\begin{equation}\label{prodmet}
d^f((x_1,s_1),(x_2,s_2))=d(x_1,x_2)+|s_1-s_2|\text{ for
}(x_1,s_1),(x_2,s_2)\in X^f.
\end{equation}

\begin{Definition}A measure-preserving flow $(S_t)_{t\in\R}$ on a standard
probability Borel space $(X,\mathcal{B},\mu)$ is {\em  mixing} if
\[\lim_{t\to\infty}\mu(S_tA\cap B)=\mu(A)\mu(B)\text{ for all }
A,B\in\mathcal{B}.\]
If for all $A,B\in\mathcal{B}$
\[\lim_{T\to\infty}\frac{1}{T}\int_0^T|\mu(S_tA\cap
B)-\mu(A)\mu(B)|\,dt=0\] then $(S_t)_{t\in\R}$ is {\em weakly
mixing}.
\end{Definition}

Of course, mixing implies weak mixing, and the following
conditions are equivalent (see \cite{Co-Fo-Si}):
\begin{itemize}
\item[(i)] $(S_t)_{t\in\R}$ is  weakly mixing;
\item[(ii)] the Cartesian product flow $(S_t\times S'_t)_{t\in\R}$ is ergodic
provided that $(S'_t)_{t\in\R}$ is an ergodic flow on a standard
 probability Borel space;
\item[(iii)] if $F:X\to\C$ is an eigenfunction corresponding to an eigenvalue $\theta \in \R$, i.e.\ $F(S_tx)=e^{i t\theta}F(x)$
then $\theta=0$ and $F$ is constant.
\end{itemize}

\begin{Definition}
A measure-preserving flow $(S_t)_{t\in\R}$ on a standard
probability Borel space is {\em mildly mixing} if its Cartesian
product with an arbitrary ergodic (finite or infinite
conservative) measure-preserving transformation remains ergodic.
\end{Definition}

Recall that a measure-preserving flow $(S'_t)_{t\in\R}$ on a
standard probability Borel space $(X',\mathcal{B}',\mu')$ is a
{\em factor} of the flow $(S_t)_{t\in\R}$ if there exists a
measurable map $\psi:X\to X'$ such that the image of $\mu$ via
$\psi$ is $\mu'$ and $\psi\circ S_t=S'_t\circ\psi$ for every
$t\in\R$. Then the flow is $(S_t)_{t\in\R}$ called an {\em
extension} of $(S'_t)_{t\in\R}$. If additionally, $\psi$ is
finite-to-one almost everywhere then $(S_t)_{t\in\R}$ a {\em
finite extension} of $(S'_t)_{t\in\R}$.

A measure-preserving flow $(S_t)_{t\in\R}$ on a standard
probability Borel space $(X,\mathcal{B},\mu)$ is {\em rigid} if
there exists a sequence $(t_n)$, $t_n\to\infty$ such that
$\mu(S_{t_n}B\triangle B)\to 0$ as $n\to\infty$ for every
$B\in\mathcal{B}$.

It is also proved in \cite{Fu-We} that a probability
measure--preserving flow $(S_t)_{t\in\R}$ on $(X,\mathcal{B},\mu)$
is mildly mixing iff $(S_t)_{t\in\R}$ has no non-trivial rigid
factor, i.e.\ \[\liminf_{t\to\infty}\mu(S_tB\triangle B)>0\text{
for every }B\in\mathcal{B}\text{ with } 0<\mu(B)<1.\] It follows
that the  mixing property of a flow implies its mild mixing which
in turn implies the weak mixing property.

Assume that $T$ is an ergodic automorphism and $f:X\to\R_+$ is in
$L^1\xbm$. It is well-known (see e.g.\ \cite{KaR}) that the
special flow $T^f$ is  weakly mixing if and only if for every
$s\in\R\setminus\{0\}$ the equation
\begin{equation}\label{kochequ}
\psi(Tx)/\psi(x)= e^{2\pi isf(x)}
\end{equation}
has no measurable solution $\psi:X\to\cir=\{z\in\C:\:|z|=1\}$.
Assume moreover that $T$ is {\em rigid}, i.e.\ for some increasing
sequence $(q_n)$,  $\mu(T^{q_n}A\cap A)\to \mu(A)$ for each
$A\in\cb$. We will make use of the following simple criterion of
weak mixing of special flows over rigid systems.

\begin{Prop}\label{cr11} Under the above assumptions
suppose additionally that  there exists $C>0$ such that
\[\left|\int_{X}e^{2\pi i sf^{(n)}(x)}\,d\mu(x)\right|\leq
C/|s|\] for every $s\neq 0$ and for all $n$ large enough. Then
(\ref{kochequ}) has no measurable solution for $s\neq0$ and
therefore the special flow $T^f$ is weakly mixing.
\end{Prop}

\begin{proof} Suppose that for some $s\neq0$ and a measurable
$\psi:X\to\cir$
$$\psi(Tx)/\psi(x)= e^{2\pi isf(x)}.$$
Then for all $k\in\Z\setminus\{0\}$ and all $n$ large enough we
have
\[\left|\int_{X}\psi^k(T^{q_n}x)\overline{\psi^k(x)}
\,d\mu(x)\right|= \left|\int_{X}e^{2\pi i
ksf^{(q_n)}(x)}\,d\mu(x)\right|\leq C/|ks|\] and since clearly
$\psi^k\circ T^{q_n}\cdot\overline{\psi^k}\to1$ in measure, when
$n\to\infty$, we obtain a contradiction.
\end{proof}

We denote by $\T^d$ the torus $\R^d/\Z^d$ which we will constantly
identify with the $d$--cube $[0,1)^d$. Let $\lambda_{\T^d}$ stand
for Lebesgue measure on $\T^d$.

A homeomorphism  $T$ of a compact topological space $X$ is called
{\em uniquely ergodic} if it admits a unique $T$--invariant
probability  Borel measure $\mu$. Then the measure-preserving
automorphism $T$ of $(X,\mu)$ is ergodic and for every continuous
function $f:X\to\C$
\begin{equation}\label{uniqerg}
\frac{1}{n}\sum_{k=0}^{n-1}f(T^kx)\to\int_X f\,d\mu\text{
uniformly in }x\in X.
\end{equation}
Recall that if $T:\T^d\to\T^d$ is the rotation by a vector
$(\alpha_1,\ldots,\alpha_d)\in\T^d$ such that
$\alpha_1,\ldots,\alpha_d,1$ are independent over $\Q$ then $T$ is
uniquely ergodic. Moreover,  using standard arguments this gives
(\ref{uniqerg})  for every Riemann integrable function
$f:\T^d\to\C$ with $\mu=\lambda_{\T^d}$.

For a real number $t$ denote by $\{t\}$ its fractional part and by
$\|t\|$ its distance to the nearest integer number.  For an
irrational $\alpha\in\T$ denote by $(q_n)$ its sequence of
denominators (see e.g.\ \cite{Ch}), that is we have
\begin{equation}\label{ulla}
\frac{1}{2q_nq_{n+1}}<\left|\alpha-\frac{p_n}{q_n}\right|
<\frac{1}{q_nq_{n+1}},
\end{equation}
where
\[\begin{array}{ccc}
q_0=1, & q_1=a_1, & q_{n+1}=a_{n+1}q_n+q_{n-1}\\
p_0=0, & p_1=1,  & p_{n+1}=a_{n+1}p_n+p_{n-1}.
\end{array}\]
Let  $[0;a_1,a_2,\dots]$ stand for the continued fraction
expansion of $\alpha$. The  rational numbers $p_n/q_n$ are called
{\em the convergents} of the continued fraction. The number
$\alpha$ is said to have {\em bounded partial quotients} if the
sequence $(a_n)$ is bounded. Then there exists a natural number
$C$ such that $\|n\alpha\|\geq 1/(C|n|)$ for every non--zero
integer $n$. It follows that $q_{s+1}\leq Cq_s$ holds for each
natural $s$.

\begin{Definition}\label{defpiec}
A function $f:\T^2\to\R$ is called a {\em piecewise
$C^r$--function} if there exist $0\leq a_1<\ldots<a_N<1$ and
$0\leq b_1<\ldots<b_M<1$ such that
$f:(a_j,a_{j+1})\times(b_k,b_{k+1})\to\R$ is of class $C^r$ and it
has a $C^r$--extension to $[a_j,a_{j+1}]\times[b_k,b_{k+1}]$ for
every $1\leq j\leq N$ and $1\leq k\leq M$, where $a_{N+1}=a_1$ and
$b_{M+1}=b_1$ and the intervals $[a_N,a_1]$ and $[b_M,b_1]$ are
meant mod~$1$.
\end{Definition}

\begin{Remark}
Modifying $f$ on a set of measure zero, if necessary, we can
always assume that $f$ is of class $C^r$ on every set
$[a_j,a_{j+1})\times[b_k,b_{k+1})$.
\end{Remark}

\section{Weak mixing}\label{secweakm}  In this section we will show weak
mixing assuming that the roof function $f:\T^2\to\R_+$ is
piecewise $C^2$  and satisfies the von Neumann condition
(\ref{vonNeumann02}) (in the following section we will establish
the absence of partial rigidity under the same assumption).  We
recall that all rotations on tori are rigid.
\begin{Lemma}[see \cite{Iw-Le-Mu}]\label{intbyparts}
Let $h:\T\to\R$ be a piecewise absolutely continuous map with $N$
discontinuities. Suppose that $h':\T\to\R$ is of bounded variation
and $|h'(x)|\geq \theta>0$ for all $x\in\T$. Then
\[\left|\int_\T e^{2\pi i h(x)}\,dx\right|\leq\frac{N}{\pi\theta}+\frac{\var h'}{2\pi\theta^2}.\]
\end{Lemma}

\begin{proof}
Suppose that $0\leq a_1<\ldots<a_N<1$ are all discontinuities of
$h$ (we set $a_{N+1}=a_1$). Using integration by parts we obtain
\begin{eqnarray*}
\int_{a_j}^{a_{j+1}}e^{2\pi i h(x)}\,dx&=&\int_{a_j}^{a_{j+1}}
\frac{1}{2\pi i h'(x)}\,d e^{2\pi i h(x)}\\ &=&
\left[\frac{e^{2\pi ih(x)}}{2\pi i
h'(x)}\right]_{a_j^+}^{a_{j+1}^-} -\int_{a_j}^{a_{j+1}}
 e^{2\pi i h(x)}\,d\frac{1}{2\pi i h'(x)}.
\end{eqnarray*}
Moreover,
\[
\left|\int_{a_j}^{a_{j+1}}
 e^{2\pi i h(x)}\,d\frac{1}{2\pi i h'(x)}\right|
\leq\frac{1}{2\pi}\var_{[a_j,a_{j+1}]}\frac{1}{h'}\leq
\frac{1}{2\pi\theta^2}\var_{[a_j,a_{j+1}]}{h'}\] and
\[\left|\left[\frac{e^{2\pi ih(x)}}{2\pi i
h'(x)}\right]_{a_j^+}^{a_{j+1}^-}\right|\leq
\frac{1}{\pi\theta}.\]
It follows that
\[\left|\int_\T e^{2\pi i h(x)}\,dx\right|\leq\sum_{j=1}^N
\left(\frac{1}{\pi\theta}+\frac{1}{2\pi\theta^2}
\var_{[a_j,a_{j+1}]}{h'}\right)=
\frac{N}{\pi\theta}+\frac{\var h'}{2\pi\theta^2}.\]
\end{proof}

\begin{Th}\label{slabemieszanie}
Let $T:\T^2\to\T^2$, $T(x,y)=(x+\alpha,y+\beta)$ be an ergodic
rotation. Suppose that $f:\T^2\to\R_+$ is a piecewise
$C^2$--function satisfying (\ref{vonNeumann02}). Then the special
flow $T^f$ is weakly mixing.
\end{Th}

\begin{proof}
Suppose that $\int_{\T^2}f_x(x,y)\,dxdy\neq 0$. The proof of the
symmetric case runs similarly. By Proposition~\ref{cr11}, it
suffices to show that there exist $C>0$ and $n_0\in\N$ such that
for every $s \neq 0$ and $n\geq n_0$ we have
$\left|\int_{\T^2}e^{2\pi i sf^{(n)}(x,y)}\,dxdy\right|\leq
C/|s|$. Since $f_x:\T^2\to\R$ is Riemann integrable and $T$ is
uniquely ergodic, $(f^{(n)})_x/n=(f_x)^{(n)}/n$ tends uniformly to
$\int_{\T^2}f_x(x,y)\,dxdy\neq 0$. Therefore there exist
$\theta>0$ and $n_0\in\N$ such that $|(f^{(n)})_x(x,y)|\geq \theta
n$ for all $(x,y)\in\T^2$ and $n\geq n_0$. Fix $n\geq n_0$ and
$y\in\T$. Since $\T\ni x\mapsto f^{(n)}(x,y)\in\R$ is a piecewise
$C^2$--function with at most $nN$ discontinuities, by
Lemma~\ref{intbyparts} applied to $f^{(n)}(\,\cdot\,,y)$,
\begin{eqnarray*}\left|\int_{\T}e^{2\pi i
sf^{(n)}(x,y)}\,dx\right|&\leq& \frac{nN}{\pi|s|\theta
n}+\frac{\var s(f^{(n)})'(\,\cdot\,,y)}{2\pi s^2\theta^2
n^2}\\&\leq&\frac{N}{\pi|s|\theta
}+\frac{\sum_{k=0}^{n-1}\|f''(\,\cdot\,,y+k\beta)\|_{C^0}}{2\pi
s\theta^2
n^2}\\
& \leq&\frac{N}{\pi|s|\theta }+\frac{\|f_{xx}\|_{C^0}}{2\pi
|s|\theta^2 n},
\end{eqnarray*}
so also
\[\left|\int_{\T^2}e^{2\pi
i sf^{(n)}(x,y)}\,dxdy\right|\leq \int_{\T}\left|\int_{\T}e^{2\pi
i sf^{(n)}(x,y)}\,dx\right|dy\leq\frac{N}{\pi|s|\theta
}+\frac{\|f_{xx}\|_{C^0}}{2\pi |s|\theta^2 n},\] which completes
the proof.
\end{proof}

\section{Absence of partial rigidity}\label{secabsrig} Let us recall that a flow
$(S_t)_{t\in\R}$ acting on a standard probability Borel space
$\xbm$ is called {\em partially rigid} if there exist $\kappa>0$
and $\R\ni r_t\to\infty$ such that $\liminf_{t\to\infty}\mu(A\cap
S_{r_t}A)\geq \kappa\mu(A)$ for each $A\in\cb$.

\begin{Th}\label{aopr}
Let $T:\T^2\to\T^2$, $T(x,y)=(x+\alpha,y+\beta)$ be an ergodic
rotation. Suppose that $f:\T^2\to\R_+$ is a piecewise
$C^1$--function satisfying (\ref{vonNeumann02}). Then the special
flow $T^f$ is not partially rigid.
\end{Th}

To prove Theorem~\ref{aopr} we will need the following.

\begin{Lemma}\label{lemkawc1}
Let $(f_n)_{n\in\N}$ be a sequence of piecewise $C^1$--functions
$f_n:\T\to\R_+$ for which there exist $0<c<C$, $0<\theta<\Theta$,
$m_0\in\N$, $N\in\N$  and finite sets $D(f_n)\subset\T$ containing
all discontinuity points of $f_n$ such that
\begin{eqnarray}
\label{zalo1}& &f_{n-1}(x)+c\leq f_{n}(x)\leq f_{n-1}(x)+C\text{
for all $n\in\N$,
$x\in\T$ ($f_0\equiv 0$)},\\
\label{zalo2}& &D(f_n)\subset D(f_{n+1})\text{ and }\#D(f_n)
\leq Nn,\\
\label{zalo3}& &\theta n\leq|f'_n(x)|\leq \Theta n\text{ for all
}n\geq n_0\text{ and }x\in\T\setminus D(f_n).
\end{eqnarray}
Then for every $t\geq 2Cn_0$ and $0<\vep<c/4$ we have
\[\lambda_\T\left(\{x\in
\T:\exists_{j\in\N}\;|f_j(x)-t|<\vep\}\right)<\frac{16C}{\theta
c^2 }(Nc+\Theta)\vep.\]
\end{Lemma}

\begin{proof}
Fix $t\geq 2Cn_0$ and $0<\vep<c/4$. Notice that, by (\ref{zalo1}),
$jc\leq f_j\leq jC$ for all $j\geq 0$. Let $J$ stand for the set
of all natural $j$ such that $|f_j(x)-t|<\vep$ for some $x\in\T$.
Then for such $j$ and $x$ we have $t+\vep>f_j(x)\geq cj\;\;\text{
and }\;\; t-\vep<f_j(x)\leq Cj$, whence
\begin{equation}\label{szaj}
t/(2C)\leq(t-\vep)/C<j<(t+\vep)/c\leq 2t/c
\end{equation}
for any $j\in J$; in particular, $J$ is finite and $j\in J$
implies \[ n_0\leq \frac{t}{2C}<j.\]

Let $\bar{j}=\max J$. Set $k:=\#D(f_{\bar{j}})\leq
N\bar{j}<2Nt/c$. The elements of $D(f_{\bar{j}})$ partition $\T$
into subintervals $I_1,\ldots,I_{k}$. Notice that for every $j\in
J$ the function $f_j$ is of class $C^1$ and strictly monotone
(because of (\ref{zalo2}) and (\ref{zalo3})) on the interval
$I_i$, $i=1,\ldots,k$.

Fix $1\leq i\leq k$. For every $j\in J$ let
$I_{i,j}=\overline{\{x\in I_i:|f_j(x)-t|<\vep\}}$. Since $f_j$ is
monotone on $I_i$, $I_{i,j}$ is an interval although it can be
empty. If $I_{i,j}=[z_1,z_2]$ is not empty then, by (\ref{zalo3})
and (\ref{szaj}),
\begin{equation}\label{malesza}
\theta j|I_{i,j}|\leq|(f_j)_-(z_2)-(f_j)_+(z_1)|\leq{2\vep}
\leq\frac{4C\vep j}{t}.
\end{equation}
Now suppose that $x\in I_{i,j}$ and $y\in I_{i,j'}$ with $j\neq
j'$. Since $x,y$ are in the same interval of continuity of $f_j$,
by (\ref{zalo3}) and (\ref{zalo1}), it follows that
\begin{equation}\label{wieksza}
\begin{aligned}
\Theta \bar{j}|y-x|&\geq\Theta j|y-x|\geq|f_j(y)-f_j(x)|\\& \geq
|f_j(y)-f_{j'}(y)|-|f_{j'}(y)-t|-|f_j(x)-t| \geq c-2\vep\geq
\frac{c}{2}.
\end{aligned}
\end{equation}
In particular, there is no overlap between $I_{i,j}$ and
$I_{i,j'}$.

Let $K_i=\{j\in J:I_{i,j}\neq\emptyset\}$ and suppose that
$s=\#K_{i}\geq 1$. Then there exist $s-1$ pairwise disjoint
subintervals $H_{l}\subset I_{i}$, $l=1,\ldots,s-1$ that are
disjoint from intervals $I_{i,j}$, $j\in K_i$ and fill up the
space between those intervals. In view of (\ref{wieksza}) and
(\ref{szaj}) we have $|H_l|\geq c/(2\bar{j}\Theta)\geq
c^2/(4t\Theta)$ for $l=1,\ldots,s-1$. Therefore, by
(\ref{malesza}) and (\ref{wieksza}), we obtain
\begin{eqnarray*}
\sum_{j\in K_i} |I_{i,j}|&\leq
&s\frac{4C\vep}{t\theta}=\frac{4C\vep}{t\theta}+\frac{16C\vep\Theta
}{c^2\theta}(s-1)\frac{c^2}{4t\Theta}\\& \leq&
\frac{4C\vep}{t\theta}+\frac{16C\vep\Theta
}{c^2\theta}\sum_{l=1}^{s-1}|H_l|\leq
\frac{4C\vep}{t\theta}+\frac{16C\vep\Theta }{c^2\theta}|I_i|.
\end{eqnarray*}
Since\[B:=\{x\in\T:\exists_{j\in\N}\;|f_j(x)-t|<\vep\}\subset
\bigcup_{i=1}^{k}\bigcup_{j\in K_i }I_{i,j},\] it follows that
\begin{eqnarray*}\lambda_\T(B)&\leq&\sum_{i=1}^{k}\sum_{j\in K_i
}|I_{i,j}|\leq\sum_{i=1}^{k}\left(\frac{4C\vep}{t\theta}+\frac{16C\vep\Theta
}{c^2\theta}|I_i|\right)\\&=&\frac{4C\vep
k}{t\theta}+\frac{16C\vep\Theta }{c^2\theta}\sum_{i=1}^{k}|I_i|
=\frac{4C\vep k }{t\theta}+\frac{16C\vep\Theta
}{c^2\theta}\\&\leq&\frac{8C\vep N }{c\theta}+\frac{16C\vep\Theta
}{c^2\theta}\leq\frac{16C\vep }{c^2\theta}(Nc+\Theta).
\end{eqnarray*}
\end{proof}

\begin{proof}[Proof of Theorem~\ref{aopr}]
Suppose that  $\int_{\T^2}f_x(x,y)\,dxdy\neq 0$. The proof of the
symmetric case runs similarly. Let $c$, $C$ be positive numbers
such that $0<c\leq f(x,y)\leq C$ for every $(x,y)\in\T^2$. Assume,
contrary to our claim, that $T^f$ is partially rigid. By Lemma~7.1
in \cite{Fr-Le}, there exist $(t_n)_{n\in\N}$, $t_n\to+\infty$ and
$0<u\leq 1$ such that for every $0<\vep<c$ we have
\begin{equation}\label{lemsztywwz}
\liminf_{n\to\infty}\lambda_{\T^2}\left(\{(x,y)\in
\T^2:\exists_{j\in\N}\;|f^{(j)}(x,y)-t_n|<\vep\}\right)\geq u.
\end{equation}
Let $0\leq a_1<\ldots<a_N<1$ and $0\leq b_1<\ldots<b_M<1$ be
points determining the lines of points of discontinuity for $f$.
Since $f_x:\T^2\to\R$ is Riemann integrable, by the unique
ergodicity of $T$, there exist $0<\theta<\Theta$ and $m_0\in\N$
such that $m\theta\leq |(f_x)^{(m)}(x,y)|\leq m\Theta$ for all
$(x,y)\in\T^2$ and $m\geq m_0$.

Take $0<\vep<\frac{c^2\theta}{32C(Nc+\Theta)}u$.  Fix $y\in\T$.
For every  $m\in\N$ let us consider the map $\T\ni x\mapsto
f^{(m)}(x,y)\in\R_+$ and set $D(f^{(m)}(\,\cdot\,
,y))=\{a_k-j\alpha:1\leq k\leq N,0\leq j<m\}$. Then
$f^{(m)}(\,\cdot\, ,y)$ is piecewise $C^1$ and its discontinuity
points are contained in $D(f^{(m)}(\,\cdot\, ,y))$. Moreover,
$D(f^{(m)}(\,\cdot\, ,y))\subset D(f^{(m+1)}(\,\cdot\, ,y))$,
$\#D(f^{(m)}(\,\cdot\, ,y))\leq Nm$ and
\[f^{(m)}(x,y)=f^{(m-1)}(x,y)+f\circ T^{m-1}(x,y)
\in f^{(m-1)}(x,y)+[c,C].\] Now an application of
Lemma~\ref{lemkawc1} to the sequence $(f^{(m)}(\,\cdot\,
,y))_{m\in\N}$ gives
\[\lambda_\T\left(\{x\in
\T:\exists_{j\in\N}\;|f^{(j)}(x,y)-t_n|<\vep\}\right)<\frac{16C}{\theta
c^2 }(Nc+\Theta)\vep<u/2\] whenever $t_n>2Cm_0$. By Fubini's
Theorem,
\begin{align*}\lambda_{\T^2}&\left(\{(x,y)\in
\T^2:\exists_{j\in\N}\;|f^{(j)}(x,y)-t_n|<\vep\}\right)\\&=\int_{\T}\lambda_\T\left(\{x\in
\T:\exists_{j\in\N}\;|f^{(j)}(x,y)-t_n|<\vep\}\right)\,dy<u/2
\end{align*}
whenever $t_n>2Cm_0$, contrary to (\ref{lemsztywwz}).
\end{proof}

\section{Weak Ratner's property}\label{secweakrat}
In this section we introduce and discuss consequences of weak
Ratner's property. Weak Ratner's property will be one more
weakening of the classical Ratner condition from \cite{Ra}. The
present idea has already been used in case $P$ is finite in
\cite{Fr-Le} and \cite{Fr-Le-Le}.

\begin{Definition} \label{defrat}
Let $(X,d)$ be a $\sigma$--compact metric space, $\mathcal{B}$ be
the $\sigma$--algebra of Borel subsets of $X$, $\mu$ a probability
Borel measure on $(X,d)$. Assume that $(S_t)_{t\in\R}$ is a flow
on $(X,{\cb},\mu)$. Let $P\subset\R\setminus\{0\}$ be a compact
subset and $t_0\in\R\setminus\{0\}$. The flow $(S_t)_{t\in\R}$ is
said to have {\em the property $\rat(t_0,P)$} if for every
$\vep>0$ and $N\in\N$ there exist $\kappa=\kappa(\vep)>0$,
$\delta=\delta(\vep,N)>0$ and a subset $Z=Z(\vep,N)\in\mathcal{B}$
with $\mu(Z)>1-\vep$  such that if $x,x'\in Z$, $x'$ is not in the
orbit of $x$ and $d(x,x')<\delta$, then there are $M=M(x,x')\geq
N$, $L=L(x,x')\geq N$ such that $L/M\geq\kappa$ and there exists
$\rho=\rho(x,x')\in P$ such that
\[\frac{\# \{n\in\Z\cap[M,M+L]:d(S_{nt_0}(x),S_{nt_0+\rho}(x'))<\vep\}}{L}>1-\vep.\]
Moreover, we say that $(S_t)_{t\in\R}$ has {\em the property
$\rat(P)$} if the set of $s\in\R$ such that the flow
$(S_t)_{t\in\R}$ has the $\rat(s,P)$--property is uncountable.
Flows with the latter property are said to have {\em weak Ratner's
property}.
\end{Definition}

\begin{Remark}
Note that the original Ratner notion of $\mathcal{H}_p$--flow,
introduced in \cite{Ra}, is equivalent to  requiring that a flow
has $\rat(p,\{-p,p\})$--property.

The  notion we introduce is different from the concept of Ratner's
property presented by Witte in \cite{Wi}. The main difference is
that Witte admits compact subsets in the centralizer of the flow
$(S_t)_{t\in\R}$ as the set of displacements. In our approach this
set is included in the flow. It should be emphasized that Witte
has used his notion to prove certain rigidity phenomena of some
translations on homogeneous space but not to study the structure
of joinings which is one of our aims.
\end{Remark}

The following result is a simple consequence of Birkhoff's Ergodic
Theorem.

\begin{Lemma}\label{uwagaerg}
Let $T:\xbm\to\xbm$ be an ergodic automorphism and
$A\in\mathcal{B}$. For every $\vep>0$, $\delta>0$ and $\kappa>0$
there exist $N=N(\vep,\delta,\kappa)\in\N$ and
$X(\vep,\delta,\kappa)\in\mathcal{B}$ with
$\mu(X(\vep,\delta,\kappa))>1-\delta$ such that for every
$M,L\in\N$ with $L\geq N$ and $L/M\geq\kappa$ we have
\[\left|\frac{1}{L}\sum_{n=M}^{M+L}\chi_A(T^nx)-\mu(A)\right|<\vep\text{
for all }x\in X(\vep,\delta,\kappa).\;\;\;\Box\]
\end{Lemma}

\begin{Remark}
If the set $P\subset\R\setminus\{0\}$ is finite then using Luzin's
theorem and Lemma~\ref{uwagaerg} one can easily show that the
$\rat(s,P)$--property does not depend on the choice of the metric
$d$ on $X$ compatible with $\mathcal{B}$. We have been unable to
decide whether for $P$ infinite (and compact) the
$\rat(s,P)$--property depends on the choice of the metric; it is
very likely that it does. This is why we are forced to put one
more assumption on $d$, see~(\ref{zalmetric}) below (see also
Remark~\ref{emmanuel} below).
\end{Remark}
We will constantly assume that $(S_t)_{t\in\R}$ satisfies the
following ``almost continuity'' condition
\begin{align}\label{zalmetric}
\begin{split}
&\text{for every $\vep>0$ there exists $X(\vep)\in\mathcal{B}$
with $\mu(X(\vep))>1-\vep$  such that}\\&\text{for every $\vep'>0$
there exists $\vep_1>0$ such
that}\\&\text{$d(S_{t}x,S_{t'}x)<\vep'$ for all $x\in X(\vep)$ and
$t,t'\in[-\vep_1,\vep_1]$.}
\end{split}
\end{align}
Notice that if $(S_t)_{t\in\R}$ is a special flow acting on a
space $Y^f$ equipped with a metric of the form (\ref{prodmet})
then (\ref{zalmetric}) holds.

We intend to prove a version of famous Ratner's theorem which
describes  the structure of ergodic joinings between a system
satisfying weak Ratner's property and an arbitrary one, see
Theorem~\ref{prpr}.

Assume that $\cs=(S_t)_{t\in\R}$ and $\ct=(T_t)_{t\in\R}$ are
ergodic flows acting on $\xbm$ and $\ycn$  respectively. By a {\em
joining} one means any $(S_t\times T_t)_{t\in\R}$--invariant
probability measure $\rho$ on $(X\times Y,\cb\ot\cc)$ with the
marginals $\mu$ and $\nu$ respectively. We then write $\rho\in
J(\cs,\ct)$. The set of ergodic joinings is denoted by
$J^e(\cs,\ct)$.

An essential step of the proof of Theorem~\ref{prpr} will be based
on the following result.

\begin{Lemma}\label{rozerg}
Let $(S_t)_{t\in\R}$ and $(T_t)_{t\in\R}$ be  ergodic flows acting
on $\xbm$ and $\ycn$  respectively and let $\rho\in J(\cs,\ct)\cap
J^e(S_1,T_1)$. Assume that $(S_t)_{t\in\R}$ and $(X,d)$ satisfy
(\ref{zalmetric}). Let $P\subset\R$ be a non-empty compact set.
Suppose that $A\in\mathcal{B}$ with $\mu(\partial A)=0$ and
$B\in\mathcal{C}$. Then for every $\vep,\delta,\kappa>0$ there
exist $N=N(\vep,\delta,\kappa)\in\N$ and
$\Theta(\vep,\delta,\kappa)\in\mathcal{B}\otimes\mathcal{C}$ with
$\rho(\Theta(\vep,\delta,\kappa))>1-\delta$ such that for every
$M,L\in\N$ with $L\geq N$ and $L/M\geq\kappa$ we have
\[\left|\frac{1}{L}\sum_{j=M}^{M+L}\chi_{S_{-p}A\times B}
(S_jx,T_jy)-\rho(S_{-p}A\times B)\right|<\vep\]
for all $(x,y)\in \Theta(\vep,\delta,\kappa)$ and $p\in P$.
\end{Lemma}

\begin{Remark}\label{emmanuel}
If in Lemma~\ref{rozerg} we take $\rho=\mu\otimes\nu$, $B=Y$ and
$\kappa=1$ then for every $\vep,\delta>0$ there exist
$N(\vep,\delta)\in\N$ and $\Theta(\vep,\delta)\in\mathcal{B}$ with
$\mu(\Theta(\vep,\delta))>1-\delta$ such that for every $L\geq
N(\vep,\delta)$  we have
\begin{equation}\label{sprzecz}
\sup_{p\in
P}\left|\frac{1}{L}\sum_{j=0}^{L}\chi_{A}(S_{j+p}x)-\mu(A)\right|
<\vep\text{
for all }x\in \Theta(\vep,\delta).
\end{equation}
 As it was pointed to
us by E. Lesigne, if we let $(S_t)$ be an arbitrary flow, and
$A\in\mathcal{B}$ be also arbitrary then (\ref{sprzecz}) fails to
be true for $P=[0,1]$. This is one more reason to justify our
additional assumption~(\ref{zalmetric}) on $(S_t)$ and $d$.
\end{Remark}

\begin{proof}[Proof of Lemma~\ref{rozerg}]
Fix $\vep,\delta,\kappa>0$. Let $V_{\epsilon}(A)=\{z\in
X:d(z,A)<\epsilon\}$. Since $\mu(\partial A)=0$, there exists
$\vep'>0$ such that $\mu(V_{\vep'}(A))-\mu(A)<\vep/4$,
\[\mu(A)-\mu((V_{\vep'}(A^c))^c)=\mu(V_{\vep'}(A^c))-\mu(A^c)<
\vep/4.\] By (\ref{zalmetric}), there exists $\vep_1>0$ such that
$d(S_{t}x,S_{t'}x)<\vep'$ for all $x\in X(\vep/4)$ and
$t,t'\in[-\vep_1,\vep_1]$. It follows that

\begin{eqnarray*}
\lefteqn{\mu\left(\bigcup_{t\in[-\vep_1,\vep_1]}S_{-t}A\right)}
\\&\leq&
\mu\left(\bigcup_{t\in[-\vep_1,\vep_1]}S_{-t}A\cap
X(\vep/4)\right)+\mu\left(\bigcup_{t\in[-\vep_1,\vep_1]}S_{-t}
A\cap X(\vep/4)^c\right)\\&\leq&
\mu\left(V_{\vep'}(A)\right)+\mu\left(
X(\vep/4)^c\right)<\mu(A)+\vep/2.
\end{eqnarray*}
Similarly
$\mu\left(\bigcup_{t\in[-\vep_1,\vep_1]}S_{-t}A^c\right)\leq
\mu(A^c)+\vep/2$, and  hence
\[\mu\left(\bigcap_{t\in[-\vep_1,\vep_1]}S_{-t}A\right)=1-
\mu\left(\bigcup_{t\in[-\vep_1,\vep_1]}S_{-t}A^c\right)\geq
1-(\mu(A^c)+\vep/2)=\mu(A)-\vep/2.\] For every $\epsilon>0$ and
$p\in \R$ set
\[I(\epsilon,p)=\bigcap_{t\in[-\epsilon,
\epsilon]}(S_{-t-p}A\times B)\text{ and
}U(\epsilon,p)=\bigcup_{t\in[-\epsilon, \epsilon]}(S_{-t-p}A\times
B).\] It follows that for every  $p\in \R$ we have
\begin{eqnarray*}\lefteqn{\rho\left(U(\vep_1,p)\right)-\rho\left(S_{-p}A\times B\right)}\\&=&
\rho\left(\bigcup_{t\in[-\vep_1,\vep_1]}(S_{-t-p}A\times
B)\setminus S_{-p}A\times B\right)
\\&=&\rho\left(\left(\bigcup_{t\in[-\vep_1,\vep_1]}S_{-t-p}A
\setminus
S_{-p}A\right)\times
B\right)\leq\mu\left(\bigcup_{t\in[-\vep_1,\vep_1]}S_{-t-p}A
\setminus
S_{-p}A\right)\\&=&\mu\left((\bigcup_{t\in[-\vep_1,\vep_1]}S_{-t}A)\setminus
A\right)<\vep/2
\end{eqnarray*}
and similarly
\begin{eqnarray*}
\rho\left(S_{-p}A\times
B\right)-\rho\left(I(\vep_1,p)\right)<\vep/2.
\end{eqnarray*}
Let $Q\subset P$ be a finite set such that $P\subset
Q+[-\vep_1/2,\vep_1/2]$. By Lemma~\ref{uwagaerg} applied to
$T_1\times S_1:(X\times Y,\rho)\to (X\times Y,\rho)$ and sets
$U(\vep_1/2,q)$, $I(\vep_1/2,q)$ for $q\in Q$, there exist
$N\in\N$ and $\Theta\subset\mathcal{B}\otimes\mathcal{C}$ with
$\rho(\Theta)>1-\delta$ such that  for every $M,L\in\N$ with
$L\geq N$ and $L/M\geq\kappa$ we have
\[\left|\frac{1}{L}\sum_{j=M}^{M+L}\chi_{U(\vep_1/2,q)}(S_jx,T_jy)-\rho(U(\vep_1/2,q))\right|<\vep/2\] and
\[\left|\frac{1}{L}\sum_{j=M}^{M+L}\chi_{I(\vep_1/2,q)}(S_jx,T_jy)-\rho(I(\vep_1/2,q))\right|<\vep/2\]
for all $(x,y)\in \Theta$ and $q\in Q$. Take $p\in P$ and choose
$q\in Q$ such that $p\in q+[-\vep_1/2,\vep_1/2]$. Then
\[I(\vep_1,p)\subset I(\vep_1/2,q)\subset S_{-p}A\times B\subset
U(\vep_1/2,q)\subset U(\vep_1,p).\] Thus
\begin{eqnarray*}\lefteqn{\frac{1}{L}\sum_{j=M}^{M+L}
\chi_{S_{-p}A\times
B}(S_jx,T_jy)\leq
\frac{1}{L}\sum_{j=M}^{M+L}\chi_{U(\vep_1/2,q)}(S_jx,T_jy)}\\
&<&\rho(U(\vep_1/2,q))+\vep/2\leq\rho(U(\vep_1,p))+\vep/2
<\rho(S_{-p}A\times B)+\vep
\end{eqnarray*}
and
\begin{eqnarray*}\lefteqn{\frac{1}{L}
\sum_{j=M}^{M+L}\chi_{S_{-p}A\times
B}(S_jx,T_jy)\geq
\frac{1}{L}\sum_{j=M}^{M+L}\chi_{I(\vep_1/2,q)}(S_jx,T_jy)}\\
&>&\rho(I(\vep_1/2,q))-\vep/2\geq\rho(I(\vep_1,p))-\vep/2>\rho(S_{-p}A\times
B)-\vep,
\end{eqnarray*}
which completes the proof.
\end{proof}

\begin{Lemma}\label{miaraoto}
For every $A\in\mathcal{B}$ there exists a set $\Upsilon\subset
(0,+\infty)$ such that $(0,+\infty)\setminus\Upsilon$ is countable
and $\mu(\partial V_{\epsilon}(A))=0$ for all
$\epsilon\in\Upsilon$.
\end{Lemma}

\begin{proof}
Note that $\partial V_{\epsilon}(A)\subset\{x\in
X:d(x,A)=\epsilon\}$ and $\{\{x\in
X:d(x,A)=\epsilon\}:\epsilon>0\}$  is a family of closed pairwise
disjoint sets. Since $\mu$ is finite, the set of all $\epsilon>0$
such that $\mu(\{x\in X:d(x,A)=\epsilon\})>0$ is countable. It
follows that $\mu(\partial V_{\epsilon}(A))>0$ for at most
countably many $\epsilon>0$.
\end{proof}
\begin{Remark}\label{istgesty} Since $(X,d)$ is a Polish space,
by the regularity
of $\mu$ and Lemma~\ref{miaraoto}, we can find $\{A_i:i\in\N\}$ a
dense family  in $(\mathcal{B},\mu)$ such that $\mu(\partial
A_i)=0$ for all $i\in\N$.
\end{Remark}
\begin{Lemma}[see the proof of Theorem~3 in
\cite{Ra}]\label{lemrat}
Let $(S_t)_{t\in\R}$ and $(T_t)_{t\in\R}$ be  ergodic flows acting
on $\xbm$ and $\ycn$ respectively and let $\rho\in J^e(\cs,\ct)$.
Suppose that there exists $U\in\mathcal{B}\otimes\mathcal{C}$ with
$\rho(U)>0$ and $\delta>0$ such that   if $(x,y)\in U$, $(x',y)\in
U$ then either $x$ and $x'$ are in the same orbit or
$d(x,x')\geq\delta$. Then $\rho$ is a finite extension of
$\nu$.\bez
\end{Lemma}

\begin{Th}\label{prpr}
Let $(X,d)$ be a $\sigma$--compact metric space, $\mathcal{B}$ be
the $\sigma$--algebra of Borel subsets of $X$, $\mu$ a probability
 Borel measure on $(X,d)$. Let $(S_t)_{t\in\R}$ be a weakly mixing flow
on the space $(X,{\cb},\mu)$ that satisfies the
$\rat(P)$--property where $P\subset\R\setminus\{0\}$ is a nonempty
compact set. Assume that $(S_t)_{t\in\R}$ and $(X,d)$ satisfy
(\ref{zalmetric}).

Let $(T_t)_{t\in\R}$ be an ergodic flow on $\ycn$ and let $\rho$
be an ergodic joining of $(S_t)_{t\in\R}$ and $(T_t)_{t\in\R}$.
Then either $\rho=\mu\otimes\nu$, or $\rho$ is a finite extension
of $\nu$.
\end{Th}

\begin{proof}
 Suppose that  $\rho\in
J^e(\cs,\ct)$ and $\rho\neq\mu\otimes\nu$. Since the flow
$(S_t\times T_t)_{t\in\R}$ is ergodic on $(X\times Y,\rho)$, we
can find $t_0\neq 0$ such that the automorphism $S_{t_0}\times
T_{t_0}:(X\times Y,\rho)\to(X\times Y,\rho)$ is ergodic and the
flow $(S_{t})_{t\in\R}$  has the $\rat(t_0,P)$--property. To
simplify notation we assume that $t_0=1$.

By Remark~\ref{istgesty}, there exist two families
$\{A_i:i\in\N\}$ and $\{B_i:i\in \N\}$ dense in
$(\mathcal{B},\mu)$ and $(\mathcal{C},\nu)$ respectively  such
that $\mu(\partial A_i)=0$ for all $i\in\N$. Let us consider the
map
\[\R\ni t\mapsto\varrho(t):=
\sum_{i,j=1}^{\infty}\frac{1}{2^{i+j}}|
\rho(S_{-t}A_i\times B_j)-\rho(A_i\times B_j)|\in\R.\]
Since
\[|\varrho(t)-\varrho(t')|\leq
\sum_{i,j=1}^{\infty}\frac{1}{2^{i+j}}|
\rho(S_{-t}A_i\times B_j)-\rho(S_{-t'}A_i\times B_j)|
\leq\sum_{i=1}^{\infty}\frac{1}{2^{i}}\mu(S_{-t}A_i\triangle
S_{-t'}A_i)\] and $\R\ni t\mapsto S_t\in Aut\xbm$ is a continuous
representation, the function $\varrho$ is continuous. Notice that
$\varrho(t)>0$ for $t\neq 0$. Indeed, if $\varrho(t)=0$ then
$\rho(S_{-t}A_i\times B_j)=\rho(A_i\times B_j)$ for all
$i,j\in\N$, and hence $\rho(S_{-t}A\times B)=\rho(A\times B)$ for
all $A\in\mathcal{B}$, $B\in\mathcal{C}$. By the ergodicity of
$S_{t}$, we obtain $\rho=\mu\otimes\nu$.

Since $P\subset\R\setminus\{0\}$ is compact, there exists $\vep>0$
such that $\varrho(p)\geq\vep$ for $p\in P$. Let $M$ be a natural
number such that $\sum_{i,j>M}1/2^{i+j}<\vep/2$. Since
\[\sum_{i,j=1}^{M}\frac{1}{2^{i+j}}|\rho(S_{-p}A_i\times B_j)-
\rho(A_i\times B_j)|\geq \vep/2\text{ for all }p\in P,\]
we have
\begin{equation}\label{mini}
\forall_{p\in P}\exists_{1\leq i,j\leq M}\,|\rho(S_{-p}A_i\times
B_j)-\rho(A_i\times B_j)|\geq\vep>0.
\end{equation}
Since $\mu(\partial(A_i))=0$, by Lemma~\ref{miaraoto}, we can
choose $0<\vep_1<\vep/8$ such that
\[\mu(V_{\vep_1}(A_i)\setminus
A_i)<\vep/2\text{  and }\mu(\partial V_{\vep_1}(A_i))=0\] for
$1\leq i\leq M$. It follows that
\begin{align}\label{szaep1}
|\rho(A_i\times B_j)-\rho(V_{\vep_1}(A_i)\times B_j)|<\vep/2,\\
|\rho(S_{-t}A_i\times B_j)-\rho(S_{-t}V_{\vep_1}(A_i)\times
B_j)|<\vep/2\nonumber
\end{align}
for all $1\leq i,j\leq M$ and $t\in\R$.

Let $\kappa:=\kappa(\vep_1)(>0)$. By Lemma~\ref{uwagaerg} applied
to the sets $V_{\vep_1}(A_i)\times B_j$ and the automorphism
$S_1\times T_1$, and Lemma~\ref{rozerg} applied to the pairs of
sets $A_i,B_j$, $i,j=1,\ldots,M$, there exist a measurable set
$U\subset X\times Y$ with $\rho(U)>3/4$ and $N\in\N$ such that if
$(x,y)\in U$, $p\in P$, $1\leq i,j\leq M$, $l\geq N$ and $l/m\geq
\kappa$ then
\begin{equation}\label{some2}
\left|\frac{1}{l}\sum_{k=m}^{m+l}\chi_{V_{\vep_1}(A_i)\times
B_j}(S_{k}x,T_{k}y)-\rho(V_{\vep_1}(A_i)\times
B_j)\right|<\frac{\vep}{8},
\end{equation}
\begin{equation}\label{some3}
\left|\frac{1}{l}\sum_{k=m}^{m+l}\chi_{S_{-p}A_i\times
B_j}(S_{k}x,T_{k}y)-\rho(S_{-p}A_i\times
B_j)\right|<\frac{\vep}{8}
\end{equation}
and similar inequalities hold for $A_i\times B_j$ for
(\ref{some2}) and $S_{-p}V_{\vep_1}(A_i)\times B_j$ for
(\ref{some3}).

 Next, by the property $\rat(1,P)$, we obtain relevant
$\delta=\delta(\vep_1,N)>0$ and $Z=Z(\vep_1,N)\in\mathcal{B}$,
$\mu(Z)>1-\vep_1$.

Now assume that $(x,y)\in U$, $(x',y)\in U$, $x,x'\in Z$ and $x'$
is not in the orbit of $x$. We claim that $d(x,x')\geq \delta$.
Suppose that, on the contrary, $d(x,x')<\delta$. Then, by the
property $\rat(1,P)$, there exist $M=M(x,x')$, $L=L(x,x')\geq N$
with $L/M\geq\kappa$ and $p=p(x,x')\in P$ such that
$(\#K_p)/L>1-\vep_1$, where
\[K_p=\{n\in\Z\cap[M,M+L]:d(S_{n}(x),S_{n+p}(x'))<\vep_1\}.\]
From (\ref{mini}), there exist $1\leq i,j\leq M$ such that
\begin{equation}\label{mini1}
|\rho(S_{-p}A_i\times B_j)-\rho(A_i\times B_j)|\geq\vep>0.
\end{equation}
If $k\in K_p$ and $S_{k+p}x'\in A_i$, then $S_{k}x\in
V_{\vep_1}(A_i)$. Hence
\begin{equation}\label{kicha1}
 \begin{aligned}
\frac{1}{L}\sum_{k=M}^{M+L}&\chi_{S_{-p}A_i\times
B_j}(S_{k}x',T_{k}y)\\
&\leq \frac{\#(\Z\cap[M,M+L]\setminus K_p)}{L}+\frac{1}{L}\sum_{k\in K_p}\chi_{A_i\times B_j}(S_{k+p}x',T_{k}y)\\
& \leq
\vep/8+\frac{1}{L}\sum_{k=M}^{M+L}\chi_{V_{\vep_1}(A_i)\times
B_j}(S_{k}x,T_{k}y).
\end{aligned}
\end{equation}
Now from (\ref{some3}), (\ref{kicha1}), (\ref{some2}) and
(\ref{szaep1}) it follows that
\begin{eqnarray*}
\rho(S_{-p}A_i\times B_j) &\leq &
\frac{1}{L}\sum_{k=M}^{M+L}\chi_{S_{-p}A_i\times
B_j}(S_{k}x',T_{k}y)+\vep/8\\
& \leq &
\vep/4+\frac{1}{L}\sum_{k=M}^{M+L}\chi_{V_{\vep_1}(A_i)\times
B_j}(S_{k}x,T_{k}y)\\
& < & \vep/2+\rho(V_{\vep_1}(A_i)\times B_j)< \vep+\rho(A_i\times
B_j).
\end{eqnarray*}
Applying similar arguments we get
\[\rho(A_i\times B_j)<
\vep+\rho(S_{-p}A_i\times B_j).\] Consequently,
\[|\rho(A_i\times B_j)-\rho(S_{-p}A_i\times B_j)|<\vep,\]
contrary to (\ref{mini1}).

In summary, we have found a measurable set
$U_1=U\cap(Z(\vep_1,N)\times Y)$ and $\delta(\vep_1,N)>0$ such
that $\rho(U_1)>3/4-\vep_1>1/2$
 and if $(x,y)\in U_1$,  $(x',y)\in U_1$ then either $x$ and
$x'$ are in the same orbit or $d(x,x')\geq\delta(\vep_1,N)$. Now
an application of Lemma~\ref{lemrat} completes the proof.
\end{proof}

\section{Weak Ratner's property for special flows}\label{secwrspf}
In this section we present techniques that will help us to prove
the weak Ratner property for special flows built over isometries.
The following is a general version of Lemma~5.2 in \cite{Fr-Le}.
We omit its proof since it is showed as in \cite{Fr-Le}.

\begin{Prop}\label{lemfundsp}
Let $(X,d)$ be a compact metric space, $\mathcal{B}$  the
$\sigma$--algebra of Borel subsets of $X$ and let $\mu$ be a
probability Borel measure on $(X,d)$. Assume that
$T:(X,\mu)\to(X,\mu)$ is an ergodic isometry and $f:X\to\R$ is a
bounded positive measurable function which is bounded away from
zero. Let $P\subset\R\setminus\{0\}$ be a nonempty compact subset.
Assume that for every $\vep>0$ and $N\in\N$ there exist
$\kappa=\kappa(\vep)>0$, $0<\delta=\delta(\vep,N)<\vep$ and
$Z=Z(\vep,N)\in\mathcal{B}$, $\mu(Z)>1-\vep$ such that if $x,y\in
Z$, $0<d(x,y)<\delta$, then there are natural numbers
$M=M(x,y)\geq N$, $L=L(x,y)\geq N$ such that $L/M\geq\kappa$ and
there exists $p=p(x,y)\in P$ such that
\[\frac{1}{L}\#\left\{M\leq n< M+L:|f^{(n)}(x)-f^{(n)}(y)-p|<\vep\right\}>1-\vep.\]
Suppose that $\gamma\in\R$ is a positive number such that the
$\gamma$--time automorphism $T^f_\gamma:X^f\to X^f$ is ergodic.
Then the special flow $T^f$ has the
$\rat(\gamma,P)$--property.\bez
\end{Prop}

\begin{Definition}
Let $0<a<b$. A sequence $(x_n)_{n\geq 0}$ taking values in
$[-R,R]\cap\Z$ ($R>0$) is called $a$--{\em sparse} if there exists
an increasing sequence $(k_m)_{m\geq 0}$, $k_0=0$, of natural
numbers such that
\begin{itemize}
\item[(i)] $x_n\neq 0$ with $n\geq 1$ if and only if $n=k_m$ for some $m\geq 1$;
\item[(ii)] $k_{m+1}-k_m\geq a$ for all $m\geq 1$;
\end{itemize}
If additionally
\begin{itemize}
\item[(iii)] $k_{m+1}-k_m\leq b$ for all $m\geq 0$.
\end{itemize}
then $(x_n)_{n\geq 0}$ is $(a,b)$--{\em sparse}.
\end{Definition}

\begin{Remark}\label{remsumsp}
If $(x_n)_{n\geq 0}$ is $a$--sparse then
$\left|\sum_{k=0}^{n-1}x_k\right|\leq R(1+n/a)$.
\end{Remark}

Let $T:X\to X$ be an isometry of a metric space $(X,d)$. Let
$f:X\to\R$ be a Borel function  and let $H=\{h_1,\ldots,h_s\}$,
$s\geq 3$, a collection of real numbers. Assume that
\begin{equation}\label{naha}
h_1,\ldots,h_{s-1}\text{ are linearly independent  over }\Q\text{
and  }h_{s-1}=h_s.
\end{equation}
Let $N_j:X\times X\to\Z$, $j=1,\ldots,s+1$, and $b:X\times X\to\R$
be Borel functions such that for some constants $R,B>0$
\begin{equation}
|N_j(x,y)|\leq R,\;j=1,\ldots,s+1\;\text{ and }|b(x,y)|\leq
B\text{ whenever }d(x,y)\leq 1.
\end{equation}
 Moreover, suppose that there exist positive constants $C_0$,
$C_1< C_2$ such that for any pair of distinct $x,y\in X$ with
$d(x,y)\leq 1/2$ we have
\begin{align}
\begin{split}\label{rozkoc}
|f^{(n)}(y)&-f^{(n)}(x)-
(\sum_{j=1}^sN^{(n)}_j(x,y)h_j)|\\
&\leq|b(x,y)N_{s+1}(x,y)-b(T^nx,T^ny)N_{s+1}(T^nx,T^ny)|
+ C_0nd(x,y),
\end{split}
\\
\label{sparse1} &\forall_{1\leq j\leq
s+1}\;(N_j(T^nx,T^ny))_{n\geq 0}\text{ is
$C_1/d(x,y)$--sparse,}\\
\label{sparse2}& \exists_{1\leq j\leq
s-2}\;(N_j(T^nx,T^ny))_{n\geq 0}\text{ is
$(C_1/d(x,y),C_2/d(x,y))$--sparse,}
\end{align}
where $N^{(n)}_j(x,y):=\sum_{k=0}^{n-1}N_j(T^kx,T^ky)$.

\begin{Lemma}\label{generalrat}
Under the above assumptions there exist $0<p_0\leq p_1$ such that
for every $\vep>0$ and $N\in\N$ there exist
$\kappa=\kappa(\vep)>0$, $0<\delta=\delta(\vep,N)<\vep$   such
that if $x,y\in X$, $0<d(x,y)<\delta$, then there are natural
numbers $M=M(x,y)\geq N$, $L=L(x,y)\geq N$ such that
$L/M\geq\kappa$ and there exists $p=p(x,y)$ with $p_0\leq |p|\leq
p_1$ such that
\[\frac{1}{L}\#\left\{M\leq n< M+L:|f^{(n)}(x)-f^{(n)}(y)-p|<\vep\right\}>1-\vep.\]
\end{Lemma}

\begin{proof}
 Let
\[H':=\left\{\sum_{j=1}^sr_j h_j:r_j\in[-R,R]\cap\Z,j=1,\ldots,s\right\}\setminus\{0\}\]
and $h:=\min\{|w|:w\in H'\}>0$.
Fix $0<\vep <\min(1/2,C_0C_1/s,h/(4s))$ and $N\geq 2$. Without
loss of generality we can assume that  $0<C_1\leq 1\leq C_0, C_2$.
Set
\begin{equation}\label{defindelkap}
\delta:=\vep^3C_1/(2C_0N)\text{ and }\kappa:=\vep/(6sC_0C_2).
\end{equation}
 Fix two distinct $x,y\in X$ with
\[d:=d(x,y)<\delta\]
and set $L=[\vep/(C_0d)]$.  Note that
\begin{equation}\label{niernal}
N/\vep^2\leq \vep/(2C_0d)\leq L\leq \vep/(C_0d)<C_1/(sd).
\end{equation}
 The number $M$ will be
chosen between $C_1/d$ and $3sC_2/d$, and we will precise its
value later. Then
\[L/M\geq \frac{\vep/(2C_0d)}{3sC_2/d}=\frac{\vep}{6sC_0C_2}=\kappa\text{ and }M\geq C_1/d\geq N.\]

By assumptions, there exists an increasing sequence $(k_m)_{m\geq
0}$, $k_0=0$ of natural numbers such that
\begin{align}
\label{zera} &\text{ $N_j(T^nx,T^ny)=0$ for $k_m< n<k_{m+1}$ and
for all
$m\geq 0$ and $j=1,\ldots,s$;}\\
\label{jedynki} &\text{  for each $m\geq 1$ there
exists $1\leq j\leq s$ with $N_j(T^{k_m}x,T^{k_m}y)\neq 0$;}\\
\label{roznkgora} &\text{ $k_{m+1}-k_m\leq C_2/d$ for $m\geq 0$;}\\
\label{roznkdol} &\text{ $k_{m+s}-k_m\geq C_1/d$ for $m\geq 1$.}
\end{align}
Since $k_{m+s}-k_m\geq C_1/d>sL$,
\begin{equation}\label{ldol}
\text{for every $m\geq 1$ there exists $m\leq m'< m+s$ such that
$k_{m'+1}-k_{m'}>L$.}
\end{equation}
We use (\ref{ldol}) for $m=s+1$ and obtain $m_1$, and  apply
again~(\ref{ldol}) for $m_1+1$ to have $s<m_1< m_2\leq m_1+s$ such
that $k_{m_i+1}-k_{m_i}>L$ for $i=1,2$. It follows that the set
\[\{(m_1,m_2)\in\N^2:s<m_1\leq 2s,\;m_1< m_2\leq m_1+s,\;
k_{m_i+1}-k_{m_i}>L,\;i=1,2\}\] is not empty. Pick a pair
$(m_1,m_2)$ from this set with the smallest $m_2-m_1$. Then
\begin{equation}\label{roznkm}
k_{m+1}-k_{m}\leq L\text{ for all }m_1<m<m_2
\end{equation} and
\begin{equation}\label{uniq}
\begin{split}&\text{for each $1\leq j\leq s$ there exists at most
one natural number $m$}\\ &\text{for which
$N_j(T^{k_{m}}x,T^{k_{m}}y)\neq 0$ and $m_1<m\leq m_2$.}
\end{split}\end{equation}
Indeed, suppose contrary to our claim that there exist $1\leq
j\leq s$ and $m_1<m'_1< m'_2\leq m_2$ such that
$N_j(T^{k_{m_i'}}x,T^{k_{m_i'}}y)\neq 0$ for $i=1,2$. Then
$m'_2-m'_1<m_2-m_1\leq s$. Since $(N_j(T^nx,T^ny))_{n\geq 0}$ is
$C_1/d$--sparse, $k_{m'_2}-k_{m'_1}\geq C_1/d>sL$. Therefore,
there exists $m'_1\leq m'<m'_2$ such that $k_{m'+1}-k_{m'}>L$,
contrary to the definition of $(m_1,m_2)$.

Take $M_1\in\{k_{m_1+1}-L,k_{m_1+1}-L+1\}$ and
$M_2\in\{k_{m_2}+1,k_{m_2}+2\}$ so that
$N_{s+1}(T^{M_i}x,T^{M_i}y)=0$ for $i=1,2$. In view of
(\ref{roznkm}) and (\ref{niernal})
\begin{equation}\label{roznemow}
M_2-M_1\leq L+2+\sum_{m_1<m<m_2}(k_{m+1}-k_m)\leq (s+1)L\leq
\frac{2s}{C_0}\frac{\vep}{d}.
\end{equation}
By (\ref{zera}) and (\ref{uniq}), for each $j=1,\ldots,s$
\[N^{(M_2)}_j(x,y)-N^{(M_1)}_j(x,y)=\sum_{m_1< m\leq m_2}N_j(T^{k_m}x,T^{k_m}y)\in[-R,R]\cap\Z.\]
Moreover, in view of  (\ref{jedynki}), there exists $1\leq j_0\leq
s-2$ such that $N_{j_0}(T^{k_{m_2}}x,T^{k_{m_2}}y)\neq 0$, hence
\[N^{(M_2)}_{j_0}(x,y)-N^{(M_1)}_{j_0}(x,y)=N_{j_0}(T^{k_{m_2}}x,T^{k_{m_2}}y)\neq 0.\]
It follows that
\[\sum_{j=1}^s(N^{(M_2)}_j(x,y)-N^{(M_1)}_j(x,y))h_j\in H'.\]
Therefore
\[\left|\sum_{j=1}^sN^{(M_2-M_1)}_j(T^{M_1}x,T^{M_1}y)h_j\right|\geq h.\]
As  $N_{s+1}(T^{M_1}x,T^{M_1}y)=N_{s+1}(T^{M_2}x,T^{M_2}y)=0$, in
view of (\ref{rozkoc}) and (\ref{roznemow}),
\begin{align*}
|f^{(M_2-M_1)}&(T^{M_1}y)-f^{(M_2-M_1)}(T^{M_1}x)-
\sum_{j=1}^sN^{(M_2-M_1)}_j(T^{M_1}x,T^{M_1}y)h_j|
\\
&\leq C_0(M_2-M_1)d<2s\vep<h/2.
\end{align*}
It follows that
\begin{align*}
\left|f^{(M_2)}(y)-f^{(M_2)}(x)-(f^{(M_1)}(y)-f^{(M_1)}(x))\right|&\\=
\left|f^{(M_2-M_1)}(T^{M_1}y)-f^{(M_2-M_1)}(T^{M_1}x)\right|&>h/2.
\end{align*}
Consequently, either for  $M=M_1$ or $M=M_2$ we have
\[|f^{(M)}(y)-f^{(M)}(x)|>h/4=:p_0>0,\]
and let $M=M_i$. Since $k_{m_i}<M<M+L-1\leq k_{m_i+1}$, by
(\ref{zera}), for all $1\leq j\leq s$
\[N_j(T^nx,T^ny)=0\text{ for all }M\leq n<M+L-1,\]
hence
\begin{equation}\label{sumzero}
N^{(n-M)}_j(T^{M}x,T^{M}y)=\sum_{M\leq k<n}N_j(T^kx,T^ky)=0\text{
for all }M\leq n<M+L.
\end{equation}
Since $s< m_1\leq 2s$ and $m_2\leq m_1+s\leq 3s$, in view of
(\ref{roznkgora}) and (\ref{roznkdol}),
\[C_1/d\leq k_{m_1}\text{ and }k_{m_2}\leq 3sC_2/d,\]
hence
\begin{equation}\label{lokM}
C_1/d\leq M\leq 3sC_2/d+2.
\end{equation}
As $N_{s+1}(T^Mx,T^My)=0$, by (\ref{rozkoc}),
\begin{align*}|f^{(M)}(y)-f^{(M)}(x)|\leq\sum_{j=1}^s\left|N^{(M)}_j(x,y)\right||h_j|+\left|b(x,y)N_{s+1}(x,y)\right|+C_0Md.
\end{align*}
Since $(N_j(T^nx,T^ny))_{n\geq 0}$ is $C_1/d$--sparse, by
Remark~\ref{remsumsp} and (\ref{lokM}),
\[\left|N^{(M)}_j(x,y)\right|\leq R(1+Md/C_1)\leq R(3sC_2/C_1+2),\]
hence
\[\left|f^{(M)}(y)-f^{(M)}(x)\right|\leq R\left((3sC_2/C_1+2)\sum_{j=1}^s|h_j|+B+3sC_0C_2+2\right)=:p_1.\]
Let $p:=f^{(M)}(y)-f^{(M)}(x)$. Then, in view of (\ref{sumzero}),
(\ref{rozkoc}) and (\ref{niernal}), for each $M\leq n<M+L$ we have
\begin{align*}
|f^{(n)}&(y)-f^{(n)}(x)-p|=\left|f^{(n)}(y)-f^{(n)}(x)-(f^{(M)}(y)-f^{(M)}(x))\right|\\
&=\left|f^{(n-M)}(T^My)-f^{(n-M)}(T^Mx)-\sum_{j=1}^sN^{(n-M)}_j(T^{M}x,T^{M}y)h_j\right|\\
&\leq
C_0(n-M)d+B\left|N_{s+1}(T^{M}x,T^{M}y)\right|+B\left|N_{s+1}(T^{n}x,T^{n}y)\right|\\&<
C_0Ld +B\left|N_{s+1}(T^{n}x,T^{n}y)\right|\leq \vep
+B\left|N_{s+1}(T^{n}x,T^{n}y)\right|.
\end{align*}
Since $(N_{s+1}(T^{n}x,T^{n}y))_{n\geq 0}$ is $C_1/d$--sparse,
\[\#\{M\leq n<M+L:N_{s+1}(T^{n}x,T^{n}y)\neq 0\}\leq dL/C_1+1.\]
It follows that
\[\#\{M\leq n<M+L:|f^{(n)}(y)-f^{(n)}(x)-p|<\vep\}\geq L-dL/C_1-1.\]
In view of (\ref{defindelkap}) and (\ref{niernal}),
\[dL/C_1+1\leq \vep^3C_1/2\cdot L/C_1+L\vep^2< \vep L.\]
Consequently,
\[\#\{M\leq
n<M+L:|f^{(n)}(y)-f^{(n)}(x)-p|<\vep\}> (1-\vep)L,\] which
completes the proof.
\end{proof}

We will consider now $T$ an isometry of a (compact) metric space
$(X,d)$ which is ergodic with respect to a probability Borel
measure $\mu$. We will assume that $(\widehat{X},\widehat{d})$ is
another metric space. Moreover, we assume that
$\pi:(\widehat{X},\widehat{d})\to(X,d)$ is a surjective function
which, in addition, is  uniformly locally isometric. More
precisely, $\pi:B_{\widehat{d}}(\widehat{x},1/2)\to
B_{{d}}(\pi(\widehat{x}),1/2)$ is a bijective isometry for every
$\widehat{x}\in\widehat{X}$. Let
$\widehat{T}:\widehat{X}\to\widehat{X}$ be an isometry of
$(\widehat{X},\widehat{d})$ such that
$\pi\circ\widehat{T}=T\circ\pi$.

\begin{Prop}\label{propklu}
Let $T:(X,\mu)\to (X,\mu)$ be an ergodic isometry of a metric
space $(X,d)$. Suppose that $f:X\to\R$ is a bounded positive Borel
function which is bounded away from zero. Let
$\widehat{f}:\widehat{X}\to\R$ given by $\widehat{f}=f\circ\pi$.
Assume that there exists a collection of real numbers
$H=\{h_1,\ldots,h_s\}$  and Borel functions $b:\widehat{X}\times
\widehat{X}\to\R$, $N_j:\widehat{X}\times \widehat{X}\to\Z$,
$j=1,\ldots,s+1$, satisfying (\ref{naha})-(\ref{sparse2}) for
$\widehat{f}$ and $\widehat{T}$. Then the special flow $T^f$
satisfies weak Ratner's property.
\end{Prop}

\begin{proof}
By Lemma~\ref{generalrat} applied to $\widehat{T}$ and
$\widehat{f}$, for every $0<\vep<1/2$ and $N\in\N$ there exist
$\kappa=\kappa(\vep)>0$, $0<\delta=\delta(\vep,N)<\vep$ such that
if $\widehat{x},\widehat{y}\in \widehat{X}$,
$0<\widehat{d}(\widehat{x},\widehat{y})<\delta$, then there are
natural numbers $M=M(\widehat{x},\widehat{y})\geq N$,
$L=L(\widehat{x},\widehat{y})\geq N$ such that $L/M\geq\kappa$ and
there exists $p=p(\widehat{x},\widehat{y})$ with $p_0\leq |p|\leq
p_1$ such that
\[\frac{1}{L}\#\left\{M\leq n< M+L:|\widehat{f}^{(n)}(\widehat{x})-\widehat{f}^{(n)}(\widehat{y})-p|<\vep\right\}>1-\vep.\]

Let $x,y\in X$ arbitrary distinct point such that $d(x,y)<\delta$.
By assumption, there are distinct $\widehat{x},\widehat{y}\in
\widehat{X}$ such that $\pi(\widehat{x})=x$, $\pi(\widehat{y})=y$
and $\widehat{d}(\widehat{x},\widehat{y})=d(x,y)<\delta$. Since
\[\widehat{f}^{(n)}(\widehat{x})-\widehat{f}^{(n)}(\widehat{y})=f^{(n)}(x)-f^{(n)}(y),\]
it follows that $T$ and $f$ verify the assumptions of
Proposition~\ref{lemfundsp}  with $P=[-p_1,-p_0]\cup[p_0,p_1]$.
This gives $\rat(t_0,P)$--property for all
$t_0\in\R\setminus\{0\}$ such that $T^f_{t_0}$ is ergodic and weak
Ratner's property follows.
\end{proof}

\section{Special flows over rotations on the two torus}\label{secconcrspf}
In this section we will deal with special flows over ergodic
rotations $T(x,y)=(x+\alpha,y+\beta)$ on $\T^2$. We will
constantly assume that both $\alpha$ and $\beta$ have bounded
partial quotients. We will consider roof functions of the form
\[f(x,y)=f_1(x)+f_2(y)+g(x,y)+\gamma h(x,y),\]
where $f_1, f_2:\T\to\R$ are piecewise $C^2$--functions which are
not continuous, $g:\T^2\to\R$ is $C^2$, $h:\T^2\to\R$ is given by
$$h(x,y)=\alpha\{y\}-(\{x\}+\alpha)[\{y\}+\beta]$$
and $\gamma\in\R$.  The function $h$ naturally appears when
considering rotations on the nil-manifold which is the quotient of
the Heisenberg group modulo its subgroup of matrices with integer
coefficients.

In order to prove weak Ratner's property for the corresponding
special flows, we will apply Proposition~\ref{propklu} in which
$X=\T^2=\R^2/\Z^2$, $\widehat{X}=\R^2$, $\pi:\R^2\to\T^2$ is
defined naturally and $\widehat{T}$ is the translation on $\R^2$
by $(\alpha,\beta)$.

\begin{Lemma}\label{lemsparse}
Let $\alpha\in\R$ be an irrational number with bounded partial
quotients. Let us consider the function $N:\R\times\R\to\Z$,
$N(x,x')=[x']-[x]$. Then there exist positive constants $C_1, C_2$
such that for any pair $x,x'\in\R$ of  points with $0<|x-x'|<1/2$
the sequence $(N(x+n\alpha,x'+n\alpha))_{n\geq 0}$ is
$(C_1/|x-x'|,C_2/|x-x'|)$--sparse.
\end{Lemma}

\begin{proof}
Since $\alpha$ has bounded partial quotients, there are constants
$C_1,C_2>0$ such that for each $m\in\N$ the lengths of intervals
$I$ in the partition of $\T$ arisen from
$0,\alpha,\ldots,(m-1)\alpha$ satisfy
\[\frac{2C_1}{m}\leq |I|\leq\frac{C_2}{2m}.\]
Suppose that $x>x'$. Then $[x+n\alpha]-[x'+n\alpha]\in\{0,1\}$ and
\[[x+n\alpha]-[x'+n\alpha]=1\text{ if and only if
}n\alpha\in[-x,-x')+\Z.\] Suppose that $n_1<n_2$ are natural
numbers such that  $n_1\alpha,n_2\alpha\in[-x,-x')+\Z$ and
$n\alpha\notin[-x,-x')+\Z$ for $n_1<n<n_2$. It follows that the
interval $[-x,-x')$ (as an interval on $\T$) contains exactly one
point of the sequence $n_1\alpha,\ldots,(n_2-1)\alpha$, hence
\[|[-x,-x')|<2\frac{C_2}{2(n_2-n_1)}.\]
Moreover, $[-x,-x')$  contains exactly two points of the sequence
$n_1\alpha,\ldots,n_2\alpha$, hence
\[|[-x,-x')|> \frac{2C_1}{n_2-n_1+1}\geq \frac{C_1}{n_2-n_1}.\]
Therefore,
\[\frac{C_1}{|x'-x|}<n_2-n_1<\frac{C_2}{|x'-x|},\]
which completes the proof.
\end{proof}

\begin{Remark}
Let us consider the function $u:\R\to\R$, $u(x)=\{x\}$. Then for
the translation $x\mapsto x+\alpha$ on $\R$
\[u^{(n)}(x)=\sum_{k=0}^{n-1}\{x+k\alpha\}=
\sum_{k=0}^{n-1}(x+k\alpha-[x+k\alpha])\]
and for distinct $x,x'\in\R$ we have
\begin{equation}\label{roznlin}
u^{(n)}(x')-u^{(n)}(x)=n(x'-x)+\sum_{k=0}^{n-1}([x'+k\alpha]-[x+k\alpha]).
\end{equation}
\end{Remark}

\begin{Remark}
Let us consider the function $\widehat{h}:\R^2\to\R$,
$\widehat{h}(x,y)=\alpha\{y\}-(\{x\}+\alpha)[\{y\}+\beta]$. (Note
that $\widehat{h}=h\circ\pi$.) Observe that
\[\widehat{h}(x,y)=
\alpha y+x[y]-(x+\alpha)[y+\beta]+[x]([\{y\}+\beta]).\] Then, for
the translation $(x,y)\mapsto (x+\alpha,y+\beta)$ on $\R^2$ we
have
\[
\widehat{h}^{(n)}(x,y)=
x[y]-(x+n\alpha)[y+n\beta]+\alpha\sum_{k=0}^{n-1}(y+k\beta)+\sum_{k=0}^{n-1}[x+k\alpha]([\{y+k\beta\}+\beta]).
\]
It follows that
\begin{align*}
\widehat{h}^{(n)}(x',y')&-\widehat{h}^{(n)}(x,y)\\
=&\;x'[y']-x[y]-(x'+n\alpha)[y'+n\beta]+(x+n\alpha)[y+n\beta]+\alpha n(y'-y)\\
&+\sum_{k=0}^{n-1}([x'+k\alpha][\{y'+k\beta\}+\beta]-[x+k\alpha][\{y+k\beta\}+\beta]),
\end{align*}
hence
\begin{align*}
\widehat{h}^{(n)}(x',y')-\widehat{h}^{(n)}(x,y)=&\; \alpha
n(y'-y)-([y+n\beta]-[y])(x'-x)\\
&+x'([y']-[y])-(x'+n\alpha)([y'+n\beta]-[y+n\beta])\\
&+\sum_{k=0}^{n-1}([x'+k\alpha]-[x+k\alpha])[\{y+k\beta\}+\beta]
\\
&+\sum_{k=0}^{n-1}[x'+k\alpha]([\{y'+k\beta\}+\beta]-
[\{y+k\beta\}+\beta]).
\end{align*}
Moreover,
\begin{eqnarray*}
\lefteqn{\sum_{k=0}^{n-1}[x'+k\alpha]([\{y'+k\beta\}+\beta]-
[\{y+k\beta\}+\beta])}\\&=&
\sum_{k=0}^{n-1}[x'+k\alpha](([y'+(k+1)\beta]-
[y+(k+1)\beta])-([y'+k\beta]-[y+k\beta]))\\
&=&\sum_{k=0}^{n-1}([x'+k\alpha]-[x'+(k+1)\alpha])
([y'+(k+1)\beta]-[y+(k+1)\beta])\\
&&+[x'+n\alpha]([y'+n\beta]-[y+n\beta])-[x']([y']-[y]).
\end{eqnarray*}
Consequently,
\begin{align}\label{roznnil}
\begin{split}
\widehat{h}^{(n)}(x',y')-\widehat{h}^{(n)}&(x,y)= \alpha
n(y'-y)-([\{y\}+n\beta])(x'-x)\\
&+\{x'\}([y']-[y])-\{x'+n\alpha\}([y'+n\beta]-[y+n\beta])\\
&-\sum_{k=0}^{n-1}[\{x'+k\alpha\}+\alpha]
([y'+(k+1)\beta]-[y+(k+1)\beta])\\
&+\sum_{k=0}^{n-1}([x'+k\alpha]-[x+k\alpha])[\{y+k\beta\}+\beta].
\end{split}
\end{align}
\end{Remark}

\begin{Th}\label{specialR}
Let  $T(x,y)=(x+\alpha,y+\beta)$, $\alpha,\beta\in[0,1)$,  be an
ergodic rotation on the torus $\T^2$such that both $\alpha$ and
$\beta$ have bounded partial quotients. Let $f:\T^2\to\R_+$ be  of
the form
\[f(x,y)=f_1(x)+f_2(y)+g(x,y)+\gamma h(x,y),\]
where $f_1, f_2:\T\to\R$ are piecewise $C^2$--functions which are
not continuous and $g:\T^2\to\R$ is $C^2$. Suppose that $f_i$ has
$s_i$ discontinuities with jumps of size $d_{i,1},\ldots,
d_{i,s_i}$ for $i=1,2$. Assume that  $d_{1,1},\ldots,
d_{1,s_1},d_{2,1}\ldots, d_{2,s_2}$ are independent over $\Q$.
Then $T^f$ satisfies weak Ratner's property in the following two
cases:
\begin{itemize}
\item[(i)] $\gamma,d_{1,1},\ldots, d_{1,s_1},d_{2,1},\ldots,
d_{2,s_2}$ are independent over $\Q$ and
$\sum_{j=1}^{s_1}d_{1,j}-\beta\gamma$ or
$\sum_{j=1}^{s_2}d_{2,j}+\alpha\gamma$ is non-zero;
\item[(ii)] $\gamma=0$.
\end{itemize}
\end{Th}

\begin{proof}
Since
\[\int_{\T^2}f_x(x,y)\,dxdy=\sum_{j=1}^{s_1}d_{1,j}-\beta\gamma\text{ and }
\int_{\T^2}f_y(x,y)\,dxdy=\sum_{j=1}^{s_2}d_{2,j}+\alpha\gamma,\]
by Theorem~\ref{slabemieszanie}, the special flow $T^f$ is weakly
mixing.

Note that every piecewise $C^2$--function $F:\T\to\R$ with $s$
discontinuities $\Delta_1,\ldots,\Delta_s$ with jumps
$d_1,\ldots,d_s$ respectively, can be represented as
\[F(x)=\widetilde{F}(x)+\sum_{j=1}^sd_j\{x-\Delta_j\},\]
where $\widetilde{F}$ is a continuous function which is piecewise
$C^2$. Therefore, we can assume that
\[f_i(x)=\sum_{j=1}^sd_{i,j}\{x-\Delta_{i,j}\}\text{ for }i=1,2\]
and $g$ is a Lipschitz function.

We proceed to the proof of (i). On $\R^2$ and $\T^2$ we will
consider the metrics
$\widehat{d}((x,y),(x',y'))=\max(|x'-x|,|y'-y|)$ and
$d((x,y),(x',y'))=\max(\|x'-x\|,\|y'-y\|)$, respectively. Then the
map $\pi:\R^2\to\T^2$, $\pi(x,y)=(x,y)+\Z^2$ is surjective and
uniformly locally isometric. Moreover, $\pi$ is equivariant for
the translation $\widehat{T}:\R^2\to\R^2$,
$\widehat{T}(x,y)=(x+\alpha,y+\beta)$ and $T$. Let
$\widehat{f}=f\circ \pi$ and $\widehat{g}=g\circ \pi$. In view of
(\ref{roznlin}) and (\ref{roznnil})
\begin{align*}
\widehat{f}^{(n)}&(x',y')-\widehat{f}^{(n)}(x,y)=\widehat{g}^{(n)}(x',y')-\widehat{g}^{(n)}(x,y)\\
&+n\left(\sum_{j=1}^{s_1}d_{1,j}-\gamma\frac{[\{y\}+n\beta]}{n}\right)(x'-x)+n\left(\sum_{j=1}^{s_2}d_{2,j}+\gamma\alpha\right)(y'-y)\\
&+\sum_{i=1,2}\left(\sum_{j=1}^{s_i}N_{i,j}^{(n)}((x,y),(x',y'))d_{i,j}+
N_i^{(n)}((x,y),(x',y'))\gamma\right)\\&-
b(\widehat{T}^n(x,y),\widehat{T}^n(x',y'))
N(\widehat{T}^n(x,y),\widehat{T}^n(x',y'))\\&
+b((x,y),(x',y'))N((x,y),(x',y')),
\end{align*}
where $N_{(\,\cdot\,)}:\R^2\times\R^2\to\Z$ and
$b:\R^2\times\R^2\to\R$ are given by
\[N_{1,j}((x,y),(x',y'))=[x'-\Delta_{1,j}]-[x-\Delta_{1,j}], \]\[N_{2,j}((x,y),(x',y'))=[y'-\Delta_{2,j}]-[y-\Delta_{2,j}],\]
\[
N_1((x,y),(x',y'))=([x']-[x])[\{y\}+\beta],\]\[N_2((x,y),(x',y'))=-[\{x'\}+\alpha]
([y'+\beta]-[y+\beta])\]
\[b((x,y),(x',y'))=\gamma\{x'\},\;\;N((x,y),(x',y'))=[y']-[y].\]
   By
Lemma~\ref{lemsparse}, there exist positive constants $C_1$, $C_2$
such that for any pair of distinct point $(x,y)$, $(x',y')$ in
$\R^2$ with $d:=\widehat{d}((x,y),(x',y'))\leq 1/2$ we have:
\begin{itemize}
\item $|N_{(\,\cdot\,)}((x,y),(x',y'))|\leq 1$ and $|b((x,y),(x',y'))|\leq |\gamma|$;
\item each sequence
$\left(N_{(\,\cdot\,)}(\widehat{T}^n(x,y),\widehat{T}^n(x',y'))\right)_{n\geq
0}$ is $C_1/d$--sparse;
\item $\left(N_{1,j}(\widehat{T}^n(x,y),\widehat{T}^n(x',y'))\right)_{n\geq 0}$ is
$(C_1/d,C_2/d)$--sparse for all $j=1,\ldots,s_1$ whenever
$|x'-x|\geq |y'-y|$;
\item $\left(N_{2,j}(\widehat{T}^n(x,y),\widehat{T}^n(x',y'))\right)_{n\geq 0}$ is
$(C_1/d,C_2/d)$--sparse for all $j=1,\ldots,s_2$ whenever
$|x'-x|\leq |y'-y|$.
\end{itemize}
Moreover, if $L$ stands for the Lipschitz constant of $g$ then
\begin{align*}|\widehat{g}^{(n)}&(x',y')-\widehat{g}^{(n)}(x,y)|\\
&+\left|n\left(\sum_{j=1}^{s_1}d_{1,j}-\frac{[\{y\}+n\beta]}{n}\right)(x'-x)
+n\left(\sum_{j=1}^{s_2}d_{2,j}+\alpha\right)(y'-y)\right|\\
&\;\;\;\;\;\;\;\;\leq nC_0\widehat{d}((x,y),(x',y')),
\end{align*}
where
\[C_0=L+\sum_{j=1}^{s_1}|d_{1,j}|+\sum_{j=1}^{s_2}|d_{2,j}|+|\alpha|+|\beta|+2.\]
 Since $d_{1,1},\ldots,
d_{1,s_1},d_{2,1},\ldots, d_{2,s_2},\gamma$ are independent over
$\Q$, the assumptions of Proposition~\ref{propklu} are verified
with $R=1$ and $B=|\gamma|$. This completes the proof of weak
Ratner's property for the special flow $T^f$ in case~(i).

The  proof in case (ii) runs as before.
\end{proof}

\section{Mild mixing}\label{MMM} Using a result from \cite{Fr-Le} we will
now show mild mixing property for the class of flows from the
previous section.
\begin{Lemma}[see \cite{Fr-Le}]\label{jmm}
Let $(S_t)_{t\in\R}$ be an ergodic flow on $\xbm$ which has finite
fibers factor property. Then the flow $(S_t)_{t\in\R}$ is mildly
mixing provided it is not partially rigid.\bez
\end{Lemma}

\begin{Th}\label{mmixing}
Under the assumptions of Theorem~\ref{specialR}, the special flow
$T^f$ is mildly mixing.
\end{Th}

\begin{proof}
In view of Theorem~\ref{prpr}, since the special flow $T^f$ has
weak Ratner's property, it is a finite extension of each of its
non--trivial factors. As $\int f_x(x,y)\,dxdy\neq 0$ or $\int
f_y(x,y)\,dxdy\neq 0$, by Theorem~\ref{aopr}, $T^f$ is not
partially rigid. An application of Lemma~\ref{jmm} completes the
proof.
\end{proof}

\begin{Example} Let us consider the roof function $f:\T^2\to\R_+$ the form
\begin{align*}&f(x,y)=a\{x\}+b\{y\}+c\text{ with }a/b\in\R\setminus\Q\text{
or }\\
&f(x,y)=a\{x\}+b\{y\}+c(\alpha\{y\}-(\{x\}+\alpha)[\{y\}+\beta])+d,
\end{align*}
where $a,b,c$ are independent over $\Q$ with $a\neq c\beta$ or
$b\neq -c\alpha$. By Theorem~\ref{mmixing}, the special flow $T^f$
is mildly mixing provided that $T(x,y)=(x+\alpha,y+\beta)$ is an
ergodic rotation on the torus $\T^2$ such that both $\alpha$ and
$\beta$ have bounded partial quotients.
\end{Example}

In the light of next section it is not however clear whether flows
from Theorem~\ref{mmixing} are not mixing. We will now show that
at least some of them are certainly not mixing. The main idea is
to find $\alpha,\beta\in\T$ so that $1,\alpha,\beta$ are
rationally independent, $\alpha$ and $\beta$ have bounded partial
quotients and the intersection of the sets of denominators of
$\alpha$ and $\beta$ are infinite. Examples of such $\alpha$ and
$\beta$ have been pointed out to us by M. Keane. Below, we present
his argument.

Let $(a_n)_{n\geq 1}$ be a palindromic sequence in
$\{1,\ldots,N\}$ (for some fixed  $N\geq 2$), i.e.\ we assume that
$(a_n)_{n\geq 1}$ has infinitely many prefixes which are
palindromes and $(a_n)$ is not eventually periodic; if in the
standard Thue-Morse sequence $01101001\ldots$ we replace $0$ by
$1$ and $1$ by $2$ the resulting sequence  is  palindromic for
$N=2$, see e.g.\ \cite{Al-Sh}. Let
\[\alpha:=[0;a_1,a_2,\ldots]\text{ and }\beta:=\{1/\alpha\}=[0;a_2,a_3,\ldots].\]
Since $\alpha$ is not quadratic irrational,   $\alpha,1/\alpha,1$
cannot be rationally dependent. Moreover, if $a_1\ldots a_{k_n+1}$
is a palindrome  then in fact
\[\alpha=[0;a_1,a_2,\ldots,a_{k_n},\ldots]\text{ and }
\beta=[0;a_{k_n},a_{k_n-1},\ldots,a_1,\ldots].\] It is classical
that
$$
[0;a_1,a_2,\ldots,a_{k_n}]=\frac{p_n}{q_n}\;\text{ and }\;
[0;a_{k_n},a_{k_n-1},\ldots,a_1]=\frac{r_n}{q_n},$$ so the
$k_n$--th denominators of $\alpha$ and $\beta$ are the same. In
this way we have obtained an infinite sequence  $(q_n)_{n\geq 1}$
for $\alpha$ and $\beta$ (each $q_n$ being the $k_n$--th
denominator of $\alpha$ and $\beta$). Setting
$f(x,y)=a\{x\}+b\{y\}+c$, by the Denjoy-Koksma inequality,
$|f^{(q_n)}(x,y)-q_n\int f\,d\mu|\leq 2(|a|+|b|)$. Since $(q_n)$
is a rigidity sequence for the ergodic rotation
$T(x,y)=(x+\alpha,y+\beta)$, by standard arguments (see
\cite{Ko0}), the special flow $T^f$ is not mixing (in fact, it is
not partially mixing, see Section~\ref{concl}).

\section{Mixing}\label{strongmixing} In this section we will show
that  von Neumann's special flows over ergodic two-dimensional
rotations can be mixing. We will make use of the following
criterion for mixing in which a partial partition of $\T$
 means a partition of a subset of $\T$.

\begin{Prop}[see Proposition 3.3 in \cite{Fa1}]\label{fayadc2}
Let $T^f$ be the special flow built over an ergodic rotation
$T:\T^2\to\T^2$, $T(x,y)=(x+\alpha,y+\beta)$ and under a piecewise
$C^2$ roof function $f:\T^2\to\R_+$. Let $(\tau_n)$, $(\vep_n)$
and $(k_n)$ be sequences of real positive numbers such that
$\tau_n\to\infty$, $\vep_n\to 0$, $k_n\to\infty$ and let
 $(\eta_n)$ be a sequence of partial partitions of $\T$, where
$\eta_n=\{C^{(n)}_i\}$ and $C^{(n)}_i$ are intervals such that
\[\sup_{C^{(n)}_i\in\eta_n}|C^{(n)}_i|\to 0\text{ and }\sum_{C^{(n)}_i\in\eta_n}|C^{(n)}_i|\to 1.\]
Suppose that there exists $n_0$ such that if $n\geq n_0$ then
\begin{itemize}
\item for any $m\in[\tau_{2n}/2,2\tau_{2n+1}]$, $y\in\T$ and
$C^{(2n)}_i\in\eta_{2n}$  the map $C^{(2n)}_i\ni x\mapsto
f^{(m)}(x,y)\in\R$ is of class $C^2$ and
\[k_{2n}\leq \inf_{x\in
C^{(2n)}_i}|f_x^{(m)}(x,y)||C^{(2n)}_i|,\]\[ \sup_{x\in
C^{(2n)}_i}|f_{xx}^{(m)}(x,y)||C^{(2n)}_i|\leq \vep_{2n}\inf_{x\in
C^{(2n)}_i}|f_x^{(m)}(x,y)|.\]
\item for any $m\in[\tau_{2n+1}/2,2\tau_{2n+2}]$, $x\in\T$ and
$C^{(2n+1)}_i\in\eta_{2n+1}$ the map $C^{(2n+1)}_i\ni y\mapsto
f^{(m)}(x,y)\in\R$ is of class $C^2$ and
\[k_{2n+1}\leq \inf_{y\in
C^{(2n+1)}_i}|f_y^{(m)}(x,y)||C^{(2n+1)}_i|,\]\[ \sup_{y\in
C^{(2n)}_i}|f_{yy}^{(m)}(x,y)||C^{(2n+1)}_i|\leq
\vep_{2n+1}\inf_{y\in C^{(2n+1)}_i}|f_{y}^{(m)}(x,y)|.\]
\end{itemize}
Then $T^f$ is mixing.\bez
\end{Prop}

\begin{Remark}
The above criterion for mixing has been formulated by Fayad
\cite{Fa1} only for $C^2$ roof functions. Nevertheless, following
word by word Fayad's  proof we obtain that the assertion holds
whenever $f$ is piecewise $C^2$.
\end{Remark}

Let $(\gamma(n))_{n\in\N}$ be an increasing sequence of positive
real numbers such that $\gamma(1)\geq 1$ and $\gamma(n)\to\infty$.
Choose a pair of irrational numbers $\alpha,\beta\in[0,1)$ such
that denoting by $(q_n)$ and $(r_n)$ the sequences of denominators
for $\alpha$ and $\beta$ respectively we have
\begin{equation}\label{alfabeta}
4\gamma(n-1)\gamma(n)q_n\leq r_n\text{ and }4\gamma(n)^2r_n\leq
q_{n+1}\text{ for all }n\geq 1.
\end{equation}
As it was observed by Yoccoz in \cite[Appendix A]{Yo} the set of
all pairs satisfying~(\ref{alfabeta}) is uncountable. Note that
the rotation $T:\T^2\to\T^2$, $T(x,y)=(x+\alpha,y+\beta)$ is
ergodic. Indeed, if $T$ is not ergodic then there exist integer
numbers $k\neq 0$, $l\neq 0$ and $m$ such that $k\alpha+l\beta=m$.
Next choose $n\in\N$ such that
\begin{equation}\label{duzen}
\gamma(n)>\max(|k|,|l|).
\end{equation}
In view of (\ref{alfabeta}),
\[|lq_n|\leq\gamma(n)q_n<4\gamma(n-1)\gamma(n)q_n\leq r_n\text{ and }
2r_n\leq \frac{q_{n+1}}{2\gamma(n)^2}\leq\frac{q_{n+1}}{2|k|}.
\]
It follows that
\[\|lq_n\beta\|\geq\|r_{n-1}\beta\|\geq\frac{1}{2r_{n}}\geq\frac{2|k|}{q_{n+1}}.\]
Moreover,
\[\|kq_n\alpha\|\leq |k|\|q_n\alpha\|\leq \frac{|k|}{q_{n+1}}.\]
Therefore
\[0=\|q_nm\|=\|q_n(k\alpha+l\beta)\|\geq\|lq_n\beta\|-\|kq_n\alpha\|\geq\frac{|k|}{q_{n+1}}>0,\]
a contradiction.

\begin{Th} \label{smixing}
Let $f:\T^2\to\R_+$ be a piecewise $C^2$--function satisfying
(\ref{vonNeumann03}). For every rotation $T:\T^2\to\T^2$,
$T(x,y)=(x+\alpha,y+\beta)$ satisfying (\ref{alfabeta}) the
special flow $T^f$ is mixing.
\end{Th}

\begin{proof}
Let $0\leq a_1<\ldots<a_N<1$ and $0\leq b_1<\ldots<b_M<1$ be
points determining the lines  of discontinuities for $f$. Since
$f_x,f_y:\T^2\to\R$ are Riemann integrable function, by the unique
ergodicity of $T$ and (\ref{vonNeumann03}), there exist $\theta>0$
and $m_0\in\N$ such that
\begin{equation}\label{pochdol}
m\theta\leq |(f_x)^{(m)}(x,y)|\text{ and }m\theta\leq
|(f_y)^{(m)}(x,y)|
\end{equation}
for all $(x,y)\in\T^2$ and $m\geq m_0$. Let
\[\Theta=\sup_{(x,y)\in\T^2}\max(|f_{xx}(x,y)|,|f_{yy}(x,y))|.\]
Then
\begin{equation}\label{pochgor}
|(f_{xx})^{(m)}(x,y)|\leq m\Theta\text{ and
}|(f_{yy})^{(m)}(x,y)|\leq m\Theta.
\end{equation}

Choose $n_0\in\N$ such that $q_{n_0},r_{n_0}\geq m_0$. Fix $n\geq
n_0$. Let  $\kappa$ stand for the partition (into intervals) of
$\T$ determined by points $a_l-j\alpha$, $1\leq l\leq N$, $0\leq
j< q_n\left\lceil\frac{q_{n+1}}{\gamma(n)q_n}\right\rceil$
($\lceil x\rceil=\min\{n\in\Z:x\leq n\}$). Set
\[\{C^{(2n)}_i\}=\eta_{2n}=\left\{I\in\kappa:|I|>\frac{1}
{\sqrt{\gamma(n)}q_n}\right\}.\] Recall that for every $1\leq
l\leq N$ the diameter of the partition $\T$ determined by points
$a_l-j\alpha$ for $0\leq j<q_n$ is bounded by
$\frac{1}{q_n}+\frac{1}{q_{n+1}}$. Since $\eta_{2n}$ is finer than
each such partition,
\[\max_{C^{(2n)}_i\in\eta_{2n}}|C^{(2n)}_i|<\frac{1}{q_n}+
\frac{1}{q_{n+1}}<\frac{2}{q_n}\to 0.\]

For every pair $l,j$, where $1\leq l\leq N$ and $0\leq j<q_n$ let
us consider the family of points
\[A_{l,j}=\left\{a_l-(j+iq_n)\alpha:
0\leq
i<\left\lceil\frac{q_{n+1}}{\gamma(n)q_n}\right\rceil\right\}.\]
Note that $\bigcup_{1\leq l\leq N}\bigcup_{0\leq j<q_n}A_{l,j}$
coincides with the set determining $\kappa$. Moreover, for all
$0\leq i,i'<\left\lceil\frac{q_{n+1}}{\gamma(n)q_n}\right\rceil$
we have
\begin{eqnarray*}
\|(a_l-(j+iq_n)\alpha)-(a_l-(j+i'q_n)\alpha)\|&=&\|(i-i')
q_n\alpha\|\leq\frac{q_{n+1}}{\gamma(n)q_n}\|q_n\alpha\|
\\&\leq&\frac{1}{\gamma(n)q_n}<\frac{1}{\sqrt{\gamma(n)}q_n}.
\end{eqnarray*}
It follows that for every pair $l,j$ there exist $0\leq
i(l,j,0),i(l,j,1)<\left\lceil\frac{q_{n+1}}{\gamma(n)q_n}
\right\rceil$
such that
\[A_{l,j}\subset
I_{l,j}:=[a_l-(j+i(l,j,0)q_n)\alpha,a_l-(j+i(l,j,1)q_n)\alpha)]\]
and $|I_{l,j}|<1/(\sqrt{\gamma(n)}q_n)$. Denote by $\kappa_1$ the
family of intervals $I\in\kappa$ such that $I\subset I_{l,j}$ for
some $1\leq l\leq N$ and $0\leq j<q_n$. Since
$|I|<1/(\sqrt{\gamma(n)}q_n)$ for every $I\in\kappa_1$, we have
$\kappa_1\subset\kappa\setminus\eta_{2n}$ and
\[\lambda_\T(\bigcup_{I\in\kappa_1}I)=\sum_{1\leq l\leq N}
\sum_{0\leq j<q_n}|I_{l,j}|<
\frac{Nq_n}{\sqrt{\gamma(n)}q_n}=\frac{N}{\sqrt{\gamma(n)}}.\]
Furthermore, the ends of every interval $I\in
\kappa\setminus\kappa_1$ are of the form
$a_l-(j+i(l,j,s)q_n)\alpha$ for some $1\leq l\leq N$, $0\leq
j<q_n$ and $s=0,1$. It follows that
$\#(\kappa\setminus\kappa_1)\leq Nq_n$.

Let $\kappa_2$ stand for the collection of all
$I\in\kappa\setminus\kappa_1$ such that
$|I|\leq\frac{1}{\sqrt{\gamma(n)}q_n}$. Since $\#\kappa_2\leq
\#(\kappa\setminus\kappa_1)\leq Nq_n$, we obtain
\[\lambda_\T\left(\bigcup_{I\in\kappa_2}I\right)\leq Nq_n\frac{1}{\sqrt{\gamma(n)}q_n}=\frac{N}{\sqrt{\gamma(n)}}.\]
By the definition of $\kappa_1$ and $\kappa_2$,
$\eta_{2n}=\kappa\setminus(\kappa_1\cup\kappa_2)$, and hence
\[\sum_i|C^{(2n)}_i|=1-\lambda_\T\left(\bigcup_{I\in\kappa_1}I\right)-\lambda_\T\left(\bigcup_{I\in\kappa_2}I\right)
\geq 1-\frac{2N}{\sqrt{\gamma(n)}}\to 1.\]

Next let us consider the partition $\kappa'$  of $\T$ determined
by points $b_l-(j+ir_n)\beta$, $1\leq l\leq N$, $0\leq j<r_n$,
$0\leq i \leq
r_n\left\lceil\frac{r_{n+1}}{\gamma(n)r_n}\right\rceil$ and set
\[\{C^{(2n+1)}_i\}=\eta_{2n+1}=\left\{I\in\kappa':|I|>\frac{1}{\sqrt{\gamma(n)}r_n}\right\}.\]
Then
\[\max_{C^{(2n+1)}_i\in\eta_{2n+1}}|C^{(2n+1)}_i|<\frac{2}{r_n}\to
0\text{ and }\sum_i|C^{(2n+1)}_i|\geq
1-\frac{2M}{\sqrt{\gamma(n)}}\to 1.\] Finally for every $n\geq
n_0$ set \[\tau_{2n}=2\gamma(n)q_{n},\;\;
\tau_{2n+1}=2\gamma(n)r_{n},\]
\[\vep_{2n}=\frac{2\Theta}{\theta q_n},\;\;
\vep_{2n+1}=\frac{2\Theta}{\theta r_n},\;\;
k_{2n}=k_{2n+1}=\theta\sqrt{\gamma(n)}.\]

Assume that $m\in[\tau_{2n}/2,2\tau_{2n+1}]$ ($n\geq n_0$) and fix
$y\in\T$. From (\ref{alfabeta}) we have
\begin{equation}\label{um}
m_0<\gamma(n)q_n\leq m\leq 4\gamma(n)r_{n}\leq
q_{n+1}/\gamma(n)\leq q_n
\left\lceil\frac{q_{n+1}}{\gamma(n)q_n}\right\rceil.
\end{equation}
Then every discontinuity of $x\mapsto f^{(m)}(x,y)$ is of the form
$a_l-j\alpha$ with $1\leq l\leq N$, $0\leq j<q_n
\lceil\frac{q_{n+1}}{\gamma(n)q_n}\rceil$, and hence
$C^{(2n)}_i\ni x\mapsto f^{(m)}(x,y)\in\R$ is of class $C^2$ for
every $C^{(2n)}_i\in\eta_{2n}$. Since
$\frac{1}{\sqrt{\gamma(n)}q_n}<|C^{(2n)}_i|<\frac{2}{q_n}$, by
(\ref{pochdol}), (\ref{pochgor}) and (\ref{um}),
\begin{equation*}
\inf_{x\in C^{(2n)}_i}|f_x^{(m)}(x,y)||C^{(2n)}_i|\geq \theta m
\frac{1}{\sqrt{\gamma(n)}q_n}\geq \theta \gamma(n)q_n
\frac{1}{\sqrt{\gamma(n)}q_n}= \theta \sqrt{\gamma(n)}=k_{2n},
\end{equation*}
\[\vep_{2n}\inf_{x\in C^{(2n)}_i}|f_x^{(m)}(x,y)|\geq \frac{2\Theta}{\theta q_n}\theta m=
\frac{2\Theta m}{ q_n}\] and
\[\sup_{x\in
C^{(2n)}_i}|f_{xx}^{(m)}(x,y)||C^{(2n)}_i|\leq\Theta
m\frac{2}{q_n}\leq \vep_{2n}\inf_{x\in
C^{(2n)}_i}|f_x^{(m)}(x,y)|.\]

Similarly, if $m\in[\tau_{2n+1}/2,2\tau_{2n+2}]$ and $n\geq n_0$
then $\gamma(n)r_n\leq m \leq r_{n+1}/\gamma(n)$. Moreover, for
every $x\in\T$ and $C^{(2n+1)}_i\in\eta_{2n+1}$ the function
$C^{(2n+1)}_i\ni y\mapsto f^{(m)}(x,y)\in\R$ is of class $C^2$ and
\begin{equation*}\label{rozcy}
k_{2n+1}\leq \inf_{y\in
C^{(2n+1)}_i}|f_y^{(m)}(x,y)||C^{(2n+1)}_i|,
\end{equation*}
\[\sup_{y\in C^{(2n)}_i}|f_{yy}^{(m)}(x,y)||C^{(2n+1)}_i|\leq
\vep_{2n+1}\inf_{y\in C^{(2n+1)}_i}|f_{y}^{(m)}(x,y)|.\] Now an
application of Proposition~\ref{fayadc2} completes the proof.
\end{proof}

\section{Remarks}\label{concl} As we have already noticed in
Section~\ref{MMM} certainly not all von Neumann's special flows
over two-dimensional rotations are mixing. As a matter of fact, if
assume that $f(x,y)=f_1(x)+f_2(y)$ (we assume tacitly that $f>0$
and $\int_{\T^2}f\,d\la_{\T^2}=1$ and we set $f_0=f-1$) and
$\alpha$ and $\beta$ have a common subsequence of denominators
then basically we will copy results from the one-dimensional case.
Indeed, the strong von Neumann's condition~(\ref{vonNeumann03}) is
reduced to~(\ref{vN0}) for $f_1$ and $f_2$ separately (and $f_i$
is piecewise $C^2$, $i=1,2$). Denote by $(q_n)$ and $(t_n)$ the
sequences of denominators of $\alpha$ and $\beta$ respectively. If
we assume additionally that $\alpha$ and $\beta$ have a common
subsequence of denominators $l_k:=q_{n_k}=t_{m_k}$ for infinitely
many $k\geq1$ then it follows from \cite{Fr-Le0} that the sequence
of centered distributions $(f^{(l_k)}_0)_\ast\to P$ weakly in the
space of probability measures on $\R$ (the probability measure $P$
is concentrated on the interval
$[-(\var\,f_1+\var\,f_2),\var\,f_1+\var\,f_2]$). Thus, by
\cite{Fr-Le0}
\begin{equation}\label{astast}
U_{T^f_{l_k}}\to \int_{\R} U_{T_t^f}\,dP(t)
\end{equation}
in the space of Markov operators on $L^2((\T^2)^f,\la_{\T^2}^f)$,
whence (again by~\cite{Fr-Le0}) $T^f$ is spectrally disjoint from
all mixing flows, which in particular  rules out the possibility
of $T^f$ being mixing; here by $U_{T^f}$ we denote the
corresponding Koopman representation: $U_{T^f_t}F=F\circ T^f_t$
for $t\in\R$. In fact,~(\ref{astast}) implies even the absence of
partial mixing for $T^f$. Indeed, recall that partial mixing means
that there exists a constant $\kappa>0$ such that
$$
\liminf_{t\to\infty}\la_{\T^2}^f(A\cap T^f_t(B))\geq\kappa
\la_{\T^2}^f(A) \la_{\T^2}^f(B)$$ for each $A,B\in\cb^f$. In terms
of Markov operators it follows that for any convergent subsequence
$U_{T^f_{s_k}}\to J$ we have $J=\kappa \Pi_{(\T^2)^f}+(1-\kappa)K$
where $\Pi_{(\T^2)f}(F)=\int_{\T^2}F\,d\la_{\T^2}$ and $K$ is
another Markov operator. Now, if we take $s_k=l_k$ we will obtain
$$\int_{\R} U_{T_t^f}\,dP(t)=\kappa
\Pi_{(\T^2)^f}+(1-\kappa)K$$ which is possible only if $\kappa=0$
(indeed,  otherwise by taking an ergodic decomposition of the
joining corresponding to $K$ we would obtain two different ergodic
decompositions of the joining corresponding to the same Markov
operator, see \cite{Fr-Le0}).

The following natural questions easily follow:

\vspace{1ex}

 1) {\em Is it possible to obtain a mixing strong von
Neumann's flow over the rotation by $(\alpha,\beta)$ if $\alpha$
and $\beta$ have a common subsequence of denominators?}

\vspace{1ex}

2) {\em Given $(\alpha,\beta)\in\T^2$ is there a large class of
piecewise $C^2$ functions  satisfying~(\ref{vonNeumann03}) for
which mixing is excluded?} It seems that such a question makes
sense even in case of smooth functions on $\T^2$. We recall that
mixing of $T^f$ is excluded whenever the sequence
$((f_0^{(n)})_\ast)_{n\geq1}$ does not converge to
$\delta_{\infty}$ in the space of probability measures on
$\R\cup\{\infty\}$, see \cite{Le}, \cite{Sch}.

Note in passing  that the weak convergence of measures
$(f_0^{(n)})_\ast\to\delta_{\infty}$  takes place for all examples
coming from Theorem~\ref{smixing}.

\vspace{1ex}

3) {\em Is it possible to obtain mixing for strong von Neumann's
flows over the rotation by $(\alpha,\beta)$ where $\alpha,\beta$
have bounded partial quotients? More specifically, is mixing
possible in the class of flows considered in
Theorem~\ref{specialR}?} If the answer to the second question is
positive then Theorem~\ref{specialR} would give the first examples
of mixing special flows over rotations having (weak) Ratner's
property. For such flows mixing of all order follows; indeed,
flows having weak Ratner's property are quasi-simple in the sense
of \cite{Ry-Th} and mixing implies mixing of all orders for such
flows \cite{Ry-Th}. Another possibility to obtain mixing of all
orders would be to show that for example if we take
$f(x,y)=a\{x\}+b\{y\}+c$ and $(\alpha,\beta)$
satisfying~(\ref{alfabeta}) then the spectrum of $U_{T^f}$ is
singular: mixing of all orders would follow from \cite{Ho}. We
recall that Fayad in \cite{Fa3} has constructed a smooth
reparametrization of a linear flow on $\T^3$ which is mixing and
has  simple singular spectrum. Such a reparametrization flow has a
representation as the special flow over a two--dimensional
rotation and under a smooth roof function.

\vspace{1ex}

Little is known about the spectrum of weak von Neumann's special
flows. It seems to be completely open whether such flows can have
an absolutely continuous component in the spectrum. This is
impossible over rotations on $\T$ (in fact, in the one dimensional
case we have even spectral disjointness with all mixing flows
\cite{Fr-Le0}). It is neither clear whether such flows can have
simple spectrum -- this remains an open problem even in the one
dimensional case.

Finally, it would be nice to decide whether there exists a weak
von Neumann's special flow over two-dimensional rotations which is
self-similar -- this is impossible for von Neumann's special flows
over rotations on the circle \cite{Fr-Le3}.

\end{document}